\documentclass[11pt,letterpaper]{amsart}
\usepackage[utf8]{inputenc}
\usepackage{amsmath}
\usepackage{amssymb}
\usepackage{amsthm}
\usepackage{dsfont}
\usepackage{amsfonts}
\usepackage{mathtools}
\usepackage{relsize,exscale}
\usepackage{indentfirst}
\usepackage{graphicx}
\usepackage{caption} 
\usepackage[hidelinks]{hyperref}
\usepackage{nicematrix}
\hypersetup{
colorlinks=true,
linkcolor=blue
}

\title[Zariski-dense Hitchin representations in uniform lattices]{Zariski-dense Hitchin representations in uniform lattices}
\author{Jacques Audibert}
\address{
Sorbonne Université and Université Paris Cité, CNRS, IMJ-PRG, F-75005 Paris, France.}
\email{audibert.j@outlook.fr}

\theoremstyle{plain}

\newtheorem{proposition}{Proposition}[section]
\newtheorem{theorem}[proposition]{Theorem}

\newtheorem{corollary}[proposition]{Corollary}
\newtheorem{lemma}[proposition]{Lemma}

\theoremstyle{definition}

\newtheorem{definition}[proposition]{Definition}

\newtheorem{remark}[proposition]{Remark}

\newcommand{\SL}{\textnormal{SL}}
\DeclareMathOperator{\SU}{SU}
\DeclareMathOperator{\GL}{GL}
\newcommand{\G}{\textnormal{G}}
\DeclareMathOperator{\A}{A}
\DeclareMathOperator{\Mat}{Mat}

\DeclareMathOperator{\N}{N}
\DeclareMathOperator{\D}{D}
\DeclareMathOperator{\M}{M}
\DeclareMathOperator{\B}{B}
\DeclareMathOperator{\rank}{rank}

\newcommand{\SO}{\textnormal{SO}}
\DeclareMathOperator{\Gal}{Gal}
\DeclareMathOperator{\Aut}{Aut}

\DeclareMathOperator{\PSL}{PSL}

\DeclareMathOperator{\K}{K}
\DeclareMathOperator{\Det}{Det}
\DeclareMathOperator{\I}{I}
\DeclareMathOperator{\Res}{Res}
\DeclareMathOperator{\Int}{Int}
\DeclareMathOperator{\Hom}{Hom}
\DeclareMathOperator{\PSp}{PSp}

\DeclareMathOperator{\Comm}{Comm}
\DeclareMathOperator{\J}{J}

\DeclareMathOperator{\Diag}{Diag}

\DeclareMathOperator{\PP}{P}

\DeclareMathOperator{\X}{X}
\DeclareMathOperator{\Ad}{Ad}
\DeclareMathOperator{\T}{T}
\DeclareMathOperator{\Nrd}{Nrd}

\DeclareMathOperator{\HH}{H}

\newcommand{\Sp}{\textnormal{Sp}}

\DeclareMathOperator{\SSS}{S}

\DeclareMathOperator{\Tr}{Tr}
\DeclareMathOperator{\MCG}{MCG}
\DeclareMathOperator{\tc}{\overline{\phantom{s}}}

\DeclareMathOperator{\Qbar}{\overline{\mathds{Q}}}
\DeclareMathOperator{\GalF}{\Gal(\overline{\mathds{Q}}/\textrm{F})}

\makeatletter
\def\@tocline#1#2#3#4#5#6#7{\relax
  \ifnum #1>\c@tocdepth 
  \else
    \par \addpenalty\@secpenalty\addvspace{#2}%
    \begingroup \hyphenpenalty\@M
    \@ifempty{#4}{%
      \@tempdima\csname r@tocindent\number#1\endcsname\relax
    }{%
      \@tempdima#4\relax
    }%
    \parindent\z@ \leftskip#3\relax \advance\leftskip\@tempdima\relax
    \rightskip\@pnumwidth plus4em \parfillskip-\@pnumwidth
    #5\leavevmode\hskip-\@tempdima
      \ifcase #1
       \or\or \hskip 1em \or \hskip 2em \else \hskip 3em \fi%
      #6\nobreak\relax
    \dotfill\hbox to\@pnumwidth{\@tocpagenum{#7}}\par
    \nobreak
    \endgroup
  \fi}
\makeatother

\begin{document}

\maketitle

\begin{abstract}
    We construct Zariski-dense surface subgroups in infinitely many commensurability classes of uniform lattices of the split real Lie groups $\SL(n,\mathds{R})$, $\Sp(2n,\mathds{R})$, $\SO(k+1,k)$, and $\G_2$. These subgroups are images of Hitchin representations. In particular, we show that every uniform lattice of $\Sp(2n,\mathds{R})$, of $\SO(k+1,k)$ with $k\equiv1,2[4]$ and of $\G_2$ contains infinitely many mapping class group orbits of Zariski-dense Hitchin representations of fixed genus. Together with \cite{Long_ZariskidensesurfaceSL4Z} and \cite{Audibert_Zariskidensesurfacegroups} it implies that all lattices of $\Sp(4,\mathds{R})$ contain a Zariski-dense surface subgroup. This paper follows \cite{Audibert_Zariskidensesurfacegroups} where we constructed Zariski-dense Hitchin representations in non-uniform lattices.
\end{abstract}

\section*{Introduction}

Let $\Lambda$ be a lattice in a semisimple Lie group $G$. We call a subgroup of $\Lambda$ a \emph{surface subgroup} if it is isomorphic to the fundamental group of a closed connected orientable surface of genus at least 2. In their celebrated work \cite{Kahn_Immersingsurfacesinhyerbolicthreemanifold}, Kahn and Markovic exhibit surface subgroups in all uniform lattices of $\SO(3,1)$. Using tools from the latter, Hamenstädt proved that when $G$ is a simple non-compact rank 1 Lie group not isomorphic to $\SO(2n,1)$, all uniform lattices of $G$ contain a surface subgroup \cite{Hamenstadt_Incompressiblesurfaceslocallysymmetricspaces}. Kahn, Labourie and Mozes \cite{Kahn_Surfacegroupsinuniformlattices} extended this result for several other semisimple Lie groups, notably the complex ones. However, their constructions do not work when $G$ is real split.\footnote{See \S 1.2 of \cite{Kahn_Surfacegroupsinuniformlattices}}

Our goal is to construct Zariski-dense surface subgroups in uniform lattices of split real Lie groups that are images of a \emph{Hitchin representation}. Our method is arithmetic and does not use Kahn-Markovic's technique.

There is a great deal of interest in constructing Zariski-dense surface subgroups in lattices. These are called $\emph{thin}$ following Sarnak \cite{Sarnak_NotesThinMatrixGroups}. In the author's previous paper \cite{Audibert_Zariskidensesurfacegroups}, we constructed Zariski-dense surface subgroups in some non-uniform lattices of split real Lie groups. It has been preceded by Long, Reid and Thistlethwaite \cite{Long_ZariskidensesurfaceSL3Z}, Long and Reid \cite{Long_ConstructingThin}, Long and Thistlethwaite \cite{Long_ZariskidensesurfaceSL4Z} \cite{Long_ZariskidensesurfaceSL2k+1Z}.\footnote{See Douba \cite{Douba_Hyperboliclatticesurfacesubgroup} for construction of Zariski-dense surface subgroups in some non-uniform lattices of $\SO(n,1)$.} The present article is an extension of the construction from \cite{Audibert_Zariskidensesurfacegroups} to the case where the lattices are uniform. This case has been substantially less treated. Nevertheless, Long and Reid proved in \cite{Long_Thinsurfacesubgroupscocompactlattices} that there are infinitely many uniform lattices of $\SL(3,\mathds{R})$ that contain the image of a Zariski-dense Hitchin representation. Our technique will also give an alternative proof of this result. 

For any $n\geq2$ denote by $\tau_{n}:\SL(2,\mathds{C})\rightarrow\SL(n,\mathds{C})$ the representation where $\SL(2,\mathds{C})$ acts on the space of homogeneous polynomials in two variables $X$ and $Y$ of degree $n-1$ as
\begin{equation*}
    \begin{pmatrix}
    a&b\\c&d
    \end{pmatrix}
    X^{n-i-1}Y^i=(aX+cY)^{n-i-1}(bX+dY)^i
\end{equation*}
for every $0\leq i\leq n-1$. Up to conjugation, this is the unique $n$-dimensional irreducible representation of $\SL(2,\mathds{C})$ over $\mathds{C}$. We also denote by $\tau_n$ the induced representation of $\PSL(2,\mathds{C})$ in $\PSL(n,\mathds{C})$.

Let $S_g$ be a closed connected orientable surface of genus $g\geq2$. We call \emph{Fuchsian representation} a representation of the form $\tau_n\circ j$ where $j:\pi_1(S_g)\to\PSL(2,\mathds{R})$ is faithful and discrete. Hitchin representations are representations of $\pi_1(S_g)$ into $\PSL(n,\mathds{R})$ that lie in the same connected component of $\Hom(\pi_1(S_g),\PSL(n,\mathds{R}))/\PSL(n,\mathds{R})$ as a Fuchsian representation. These components are called \emph{Hitchin components}. They are the prototypical example of a higher rank Teichmüller space \cite{Wienhard_InvitationHigherTeichmuller}, i.e. a connected component of the character variety of $\pi_1(S_g)$ to $G$ that contains only discrete and faithful representations, as shown in Labourie \cite{Labourie_Anosovflowssurfaceandcurves} and Fock-Goncharov \cite{Fock_ModuliSpacesofLocalsystems}.

Let $G$ be $\SL(n,\mathds{R})$, $\SO(k+1,k)$, $\Sp(2n,\mathds{R})$ or $\G_2$. The greater part of the present article is dedicated to classifying lattices of $G$ that contain the image of a Fuchsian representation.

\begin{theorem}
\label{theoclassificationfuchsian}
For every lattice $\Lambda$ of $G$ listed in Table \ref{Table1}, there exists $g\geq2$ such that $\Lambda$ contains the image of a Fuchsian representation of $\pi_1(S_g)$. Furthermore, up to conjugation and commensurability, these are the only uniform lattices of $G$ that contain the image of a Fuchsian representation.

\begin{table}[htbp]
\centering
\begin{tabular}{|c|c|c|}
    \hline    
    $G$& $n$ or $k$ & $\Lambda$ \\ [0.2ex]
    \hline\hline
    $\SO(k+1,k)$& $k\geq2$, $k\equiv1,2[4]$ & Every uniform lattice \\ [0.2ex]
    \hline
    $\SO(k+1,k)$ & $k\geq 3$, $k\equiv0,3[4]$ & $\Lambda_F$ \\ [0.2ex]
    \hline
    $\emph{\textbf{G}}_2(\mathds{R})$ &  & Every uniform lattice \\ [0.2ex]
    \hline
    $\Sp(2n,\mathds{R})$ & $n\geq2$ & Every uniform lattice \\ [0.2ex]
    \hline
    $\SL(2k+1,\mathds{R})$ & $k\geq1$  & $\SU(\I_{2k+1},\sigma;\mathcal{O}_F[\sqrt{d}])$ \\ [0.2ex]
    \hline
    $\SL(2n,\mathds{R})$ & $n\geq2$  & $\SU(\I_n,\overline{\phantom{s}}\otimes\sigma;\mathcal{O}\otimes\mathcal{O}_F[\sqrt{d}])$ \\ [0.2ex]
    \hline
\end{tabular}
\caption{Here $F\neq\mathds{Q}$ is any totally real number field, $\mathcal{O}_F$ is its ring of integers and $d$ any element of $\mathcal{O}_F$ which is positive at exactly one real place of $F$. Moreover $\mathcal{O}$ is any order of a quaternion algebra over $F$ which splits exactly at the real place of $F$ where $d$ is positive. See Section \ref{subsectionexamplesarithmeticsubgroups} for notations.}
    \label{Table1}
\end{table}
\end{theorem}

Theorem \ref{theoclassificationfuchsian} is shown using nonabelian Galois cohomology. The latter allows to reformulate the question in number theoretic terms. When $G=\SO(k+1,k)$ for instance, the problem amounts to the classification of quadratic forms over algebraic number fields.

Let $\Lambda$ be a uniform lattice of $G$ which contains the image of a Fuchsian representation $\rho:\pi_1(S_g)\to G$. The image of $\rho$ is not Zariski-dense in $G$, but we can deform $\rho$ using a ``bending" procedure as introduced by Johnson and Millson \cite{Johnson_DeformationSpacesCompactHyperbolicManifolds}.
Since $\Lambda$ is uniform, the image of any simple closed curve on $S_g$ has a centralizer in $\Lambda$ of rank as big as possible, see Lemma \ref{centralizercocompactlattice}. This allows to bend $\rho$ to a new representation which lies in $\Lambda$ and is Zariski-dense in $G$. It is more involved to prove that the bending procedure provides infinitely many orbits of Hitchin representations under $\MCG(S_g)$, the mapping class group of $S_g$. This uses Weisfeiler's Strong Approximation Theorem \cite{Weisfeiler_StrongapproximationZariskidensesubgroups} (Theorem \ref{theostrongapproximation}).

\begin{theorem}
\label{theoprincipalthinHitchin}
For every lattice $\Lambda$ of $G$ listed in Table \ref{Table1} there exists $g\geq2$ such that $\Lambda$ contains the image of infinitely many $\MCG(S_g)$-orbits of Zariski-dense Hitchin representations of $\pi_1(S_g)$.
\end{theorem}

For a given lattice $\Lambda$, the genus $g$ in Theorem \ref{theoprincipalthinHitchin} agrees with the genus $g$ in Theorem \ref{theoclassificationfuchsian}. Together with \cite{Long_ZariskidensesurfaceSL4Z} for $\Sp(4,\mathds{Z})$ and \cite{Audibert_Zariskidensesurfacegroups} for the remaining non-uniform lattices, Theorem \ref{theoprincipalthinHitchin} implies the following.

\begin{corollary}
All lattices of $\Sp(4,\mathds{R})$ contain a Zariski-dense Hitchin representation.
\end{corollary}

\noindent
\textbf{Organisation of the paper.} In Section 1 we give the necessary background in Galois cohomology and arithmetic subgroups. Section 2 is dedicated to recalling results from the author's previous paper \cite{Audibert_Zariskidensesurfacegroups}. We prove Theorem \ref{theoclassificationfuchsian} in Sections 3 and 4 for the odd and even dimensional case respectively. Finally, in Section 5 we prove Theorem \ref{theoprincipalthinHitchin}.\\

\noindent
\textbf{Notation.} Through out the whole paper, $F$ denotes an algebraic number field, $V_F$ its set of archimedean places and $\mathcal{O}_F$ its ring of integers
\\

\noindent
\textbf{Acknowledgement.} The author is grateful to Andrés Sambarino for his feedback on this article and for his guidance through the last three years, which this paper is the result of. The author thanks Xenia Flamm for careful reading of this article.

\tableofcontents

\section{Background}

\subsection{Galois cohomology}

Let $\mathcal{G}$ be a topological group acting continuously on a discrete group $M$ which is not necessarily abelian. A \emph{1-cocycle} is a continuous map $\zeta:\mathcal{G}\to M$ that satisfies
$$\zeta=\zeta(s)s(\zeta(t))$$ for all $s,t\in \mathcal{G}$.
We say that two 1-cocycles $\zeta$ and $\zeta'$ are \emph{equivalent} if there exists an element $m\in M$ such that $$\zeta'(s)=m^{-1}\zeta(s)s(m)$$ for all $s\in \mathcal{G}$. We denote by $\HH^1(\mathcal{G},M)$ the set of equivalence classes of $1$-cocycles and call it \emph{the first cohomology set of $\mathcal{G}$ (with coefficients in $M$)}. It is not a group in general. See \S 1.3.2 of Platonov and Rapinchuck's book \cite{Platonov_AlgebraicgroupsNumbertheory} for more details.

Let $L/F$ be a field extension and $\G$ be an $F$-algebraic group. An \emph{$L/F$-form of $\G$} is an $F$-algebraic group $\HH$ that is isomorphic to $\G$ over $L$. If $L$ is the separable closure of $F$, $L/F$-forms of $\G$ are also called \emph{$F$-forms of $\G$}.

Suppose $L/F$ is Galois. All $L/F$-forms define $1$-cocycles as follows. Let $\HH$ be an $L/F$-form of $\G$ and $\Phi:\HH\to\G$ be an $L$-defined isomorphism. The map
\begin{align*}
    \Gal(L/F)&\to\Aut(\G(L))\\
    \sigma&\mapsto\Phi\circ\sigma\circ\Phi^{-1}\circ\sigma^{-1}
\end{align*}
is a 1-cocycle.
\begin{theorem}[Theorem 2.9 in \cite{Platonov_AlgebraicgroupsNumbertheory}]
\label{theoremcocycleandforms}
Assuming $L/F$ is Galois, this defines a bijection between the set of isomorphism classes of $L/F$-forms of $\G$ and $\HH^1(\Gal(L/F),\Aut(\G(L)))$.
\end{theorem}
For $\zeta\in\HH^1(\Gal(L/F),\Aut(\G(L)))$, we denote by $\prescript{}{\zeta}{\G}$ the corresponding $L/F$-form. Note that
$$\prescript{}{\zeta}{\G}(F)=\{g\in\G(L)\mid\zeta(\sigma)\circ\sigma(g)=g\ \forall\sigma\in\Gal(L/F)\}<\G(L).$$If a $1$-cocycle has values in $\Int(\G(L))$, then we say that this cocycle is \emph{inner}.

\begin{theorem}[Hilbert 90]
\label{Hilbert90}
Let $F$ be a number field. For any $n\geq1$ $$\HH^1(\Gal(\Qbar/F),\GL_n(\Qbar))=\{1\}.$$
\end{theorem}

For any $1$-cocycle $\zeta:\Gal(\Qbar/F)\to\GL_n(\Qbar)$ the proof of Theorem Hilbert 90 gives an algorithm to construct a matrix $S\in\GL_n(\Qbar)$ that satisfies $\zeta(\sigma)=S^{-1}\sigma(S)$ for all $\sigma\in\Gal(\Qbar/F)$, see Lemma 2.2 in Chapter 2 of \cite{Platonov_AlgebraicgroupsNumbertheory}.

\subsection{Definition of arithmetic subgroups}

If $G$ is a Lie group, we denote by $\Ad:G\to G/Z(G)$ the adjoint representation.

\begin{definition}
Let $G$ be a semisimple Lie group with no compact factors and let $\Gamma$ be a lattice in $G$. We say that $\Gamma<G$ is an
\emph{arithmetic subgroup} if there exists a semisimple $\mathds{Q}$-algebraic group $\HH$ and a surjective homomorphism $$\phi:\HH(\mathds{R})_0\to \Ad(G_0)$$ with compact kernel such that $\phi(\HH(\mathds{Z})\cap \HH(\mathds{R})_0)$ and $\Ad(\Gamma\cap G_0)$ are commensurable.
\end{definition}

When $G$ is simple, all arithmetic subgroups arise from the following construction.

\begin{proposition}[Corollary 5.5.16 in \cite{Morris_IntroductionArithmeticGroups}]
\label{propGsimplearithmetic}
Let $G$ be a non compact simple Lie group and let $\Gamma$ be an arithmetic subgroup of $G$. Then there exists a number field $F$ and an $F$-algebraic group $\HH$ such that

\begin{itemize}
\setlength\itemsep{0cm}
    \item there is $v\in V_F$ and an isogeny $$\phi:\HH(F_v)_0\to\Ad(G_0)$$ such that $\phi(\HH(\mathcal{O}_F)\cap\HH(F_v)_0)$ is commensurable with $\Ad(\Gamma\cap G_0)$ and
    \item for all $w\in V_F$ not $v$, $\HH(F_w)$ is compact with the euclidean topology.
\end{itemize}
\end{proposition}

For an examples of arithmetic groups, see Morris's book \cite{Morris_IntroductionArithmeticGroups} or \S \ref{subsectionexamplesarithmeticsubgroups}. The following proposition is a key ingredient of the proof of Theorem \ref{theoclassificationfuchsian}.

\begin{proposition}
\label{propmorphismofalgebraicgroups}
Let $f:\G\to\HH$ be a homomorphism of algebraic groups over a number field $F$. Up to commensurability, $f(\G(\mathcal{O}_F))<\HH(\mathcal{O}_F)$.
\end{proposition}

The proof can be found in Milne \cite{Milne_LieAlgebrasAlgebraicGroupsLieGroups} Proposition 5.2 of Appendix A when $F=\mathds{Q}$. It can be adapted to any number field $F$ using the restriction of scalars (see \S 10.3 in \cite{Maclachlan_ArithmeticHyperbolic3Manifolds}).

\begin{proposition}
\label{propmorphismarithmeticlattices}
Let $F$ be an algebraic number field and pick $v\in V_F$. Let $\HH$ be a connected $F$-algebraic group and $G<\HH(F_v)$ be a Zariski-closed semisimple subgroup. Let $\Gamma<G$ be a lattice such that $\Gamma<\HH(\mathcal{O}_F)$. Then there exists an $F$-algebraic subgroup $\G$ of $\HH$ such that
\begin{itemize}
\setlength\itemsep{0cm}
    \item $\G(F_v)=G$,
    \item for all $w\in V_F$ not $v$, $\G(F_w)$ is compact with the euclidean topology and
    \item $\Gamma$ is commensurable with $\G(\mathcal{O}_F)$.
\end{itemize}
\end{proposition}

\begin{proof}
The group $G\cap\HH(F)$ contains $\Gamma$ so is Zariski-dense in $G$. The proof of Proposition 5.1.8 in \cite{Morris_IntroductionArithmeticGroups} implies that there exists an $F$-algebraic subgroup $\G$ of $\HH$ satisfying $\G(F_v)=G$. For all $w\in V_F$ not $v$, $\G(F_w)$ is a closed subgroup of the compact group $\HH(F_w)$, hence is compact. Finally, since $\G(F_w)$ is compact for all archimedean places of $F$ except one, $\G(\mathcal{O}_F)$ is a lattice in $G$. It contains $\Gamma$ up to finite index which implies that it is commensurable to $\Gamma$.
\end{proof}

\subsection{Examples of arithmetic subgroups}
\label{subsectionexamplesarithmeticsubgroups}

We now present some examples of arithmetic subgroups that will appear later.

Let $F$ be a totally real number field and $\mathcal{O}_F$ its ring of integers. Let $d\in\mathcal{O}_F$ be positive at exactly one infinite place of $F$. Let $\sigma\in\Gal(F(\sqrt{d})/F)$ be non-trivial. Then
\begin{equation*}
    \SU(\I_n,\sigma;\mathcal{O}_F[\sqrt{d}])=\{\M\in\SL(n,\mathcal{O}_F[\sqrt{d}])\ |\ \sigma(\M)^\top \M= \I_n\}
\end{equation*} is an arithmetic subgroup of $\SL(n,\mathds{R})$.
It is non-uniform if and only if $F=\mathds{Q}$, see \S18.5 in \cite{Morris_IntroductionArithmeticGroups}.

\begin{definition}
A \emph{quaternion algebra} $A$ over a field $k$ is a 4-dimensional unital associative algebra which admits a basis $\{1,i,j,ij\}$ satisfying $i^2=a\in k$, $j^2=b\in k$ and $ij=-ji$.
\end{definition}

A quaternion algebra is uniquely determined by the choice of $a$ and $b$ and we denote it by $(a,b)_{k}$. For instance, the quaternion algebra $(-1,-1)_{\mathds{R}}$ is the Hamiltonian quaternions and is denoted by $\mathcal{H}$. Every quaternion algebra has a canonical involution called \emph{conjugation} and is defined by
$$\overline{\phantom{s}}:x_0+x_1i+x_2j+x_3ij\mapsto x_0-x_1i-x_2j-x_3ij.$$ The \emph{norm} of $A$ is $\Nrd:A\to k,\ x\mapsto x\overline{x}.$

Let $A$ be a quaternion algebra over a number field $F$. An \emph{order} of $A$ is a finitely generated $\mathcal{O}_F$-submodule of $A$ containing 1 which generates $A$ as a vector space and which is a subring of $A$. We say that $A$ \emph{splits} at a place $v$ of $F$ if $A\otimes_{F}F_v\simeq\M_2(F_v)$ and \emph{ramifies} at $v$ otherwise.

Let $A$ be a quaternion algebra over a totally real number field. Suppose that $A$ splits exactly at one real place $v$ of $F$. Let $\mathcal{O}$ be an order of $A$. Then $$\mathcal{O}^1=\{x\in\mathcal{O}\mid\Nrd(x)=1\}$$ is an arithmetic subgroup of $(A\otimes_{F}F_v)^1\simeq\SL_2(\mathds{R})$.

A matrix $\B\in\M_n(A)$ is said to be \emph{$\overline{\phantom{s}}$-Hermitian} if $\overline{\B}^\top=\B$, where the conjugation is applied to each entry of $\B$. Then
$$\SU(\B,\overline{\phantom{s}};\mathcal{O})=\{\M\in\SL(n,\mathcal{O})\mid\overline{\M}^\top\B\M=\B\}$$is an arithmetic subgroup of $\Sp(2n,\mathds{R})$. It is non-uniform if and only if $F=\mathds{Q}$, see \S18.5 in \cite{Morris_IntroductionArithmeticGroups}.

Let $d\in\mathcal{O}_F$ be positive exactly at the real place $v$. Let $\sigma\in\Gal(F(\sqrt{d})/F)$ be non-trivial. Denote by $$\sigma^*:\mathcal{O}\otimes_{\mathcal{O}_F}\mathcal{O}_F[\sqrt{d}]\to\mathcal{O}\otimes_{\mathcal{O}_F}\mathcal{O}_F[\sqrt{d}],\ x\otimes \lambda\to x\otimes\sigma(\lambda).$$ We also denote by $\sigma^*$ the involution induced by $\sigma$ on $\M_n(\mathcal{O}\otimes\mathcal{O}_F[\sqrt{d}])$ by applying $\sigma^*$ to each entry. Then
\begin{equation*}
    \SU(\I_n,\tc\otimes\sigma;\mathcal{O}\otimes\mathcal{O}_F[\sqrt{d}])=\{\M\in\SL(n,\mathcal{O}\otimes_{\mathcal{O}_F}\mathcal{O}_F[\sqrt{d}])\ |\ \sigma^*(\overline{\M})^\top\M=\I_n\}
\end{equation*}
is an arithmetic subgroup of $\SL(2n,\mathds{R})$. We emphasize that here the transposition is applied to $\M$ considered as an $n$-by-$n$ matrix.
It is non-uniform if and only if $F=\mathds{Q}$, see \S18.5 in \cite{Morris_IntroductionArithmeticGroups}.

Finally, we introduce a notation that was used in Table \ref{Table1}. Let $k\equiv0,3[4]$ and $F$ be any totally real number field with ring of integers $\mathcal{O}_F$. Let $q$ be non-degenerate quadratic form of rank $2k+1$ over $F$ that has trivial discriminant, trivial Hasse invariant at each finite place of $F$, which is positive definite at all real places of $F$ except one where it has signature equal to $(k+1,k)$ if $k\equiv0[4]$ or to $(k,k+1)$ if $k\equiv3[4]$. Denote by $$\Lambda_F=\SO(q,\mathcal{O}_F)$$ the subgroup of $\SO(k+1,k)$. It is a lattice of $\SO(k+1,k)$ which is non-uniform if and only if $F=\mathds{Q}$. Up to wide commensurability, it does not depend on the choice of the quadratic form.

\subsection{Classification of arithmetic subgroups}

Two subgroups of $G$ are said to be \emph{widely commensurable} if a conjugate of one is commensurable to the other.

\begin{lemma}
\label{lemmaelementnumberfield}
Let $F$ be a totally real number field and $V_F$ the set of its archimedean places. Fix a subset $V\subset V_F$. There exists $\lambda\in F$ such that for all $\iota\in V$, $\iota(\lambda)<0$ and for all $\iota\not\in V$ $\iota(\lambda)>0$.
\end{lemma}

\begin{proof}
First suppose that $V=\{\sigma\}$. We use $\sigma$ to identify $F$ with a subfield of $\Qbar$. There is an element $\alpha_0\in\Qbar$ such that $F=\mathds{Q}(\alpha_0)$. Denote $\alpha_1,..,\alpha_m$ its conjugates. Up to relabeling, we can suppose that
$$\alpha_0<\alpha_1<...<\alpha_m.$$ Let $r\in\mathds{Q}$ such that $\alpha_0<r<\alpha_1$. Then we let $\lambda=\alpha_0-r$. Then $\lambda<0$ and the conjugates of $\lambda$ are $\alpha_i-r$ which are all positive.

Let $V$ be any subset of $V_F$. For all $\sigma\in V$ let $\lambda_{\sigma}$ be an element of $F$ such that $\sigma(\lambda_{\sigma})<0$ and for all $\iota\not\in V_F\setminus\{\sigma\}$, $\iota(\lambda_{\sigma})>0$. Let $$\lambda=\prod_{\sigma\in V}\lambda_{\sigma}.$$ Then for all $\iota\in V$ we have $\iota(\lambda)<0$ and for all $\iota\not\in V$ $\iota(\lambda)>0$.
\end{proof}

\begin{proposition}
\label{propclassificationlatticeSO}
Let $n=2k+1\geq5$ be odd. Uniform lattices of $\SO(k+1,k)$ are widely commensurable with $\SO(\B,\mathcal{O}_F)$ for $\mathcal{O}_F$ the ring of integers of a totally real number field $F\neq\mathds{Q}$ and $\B\in\SL(n,F)$ a symmetric matrix satisfying:
\begin{itemize}
\setlength\itemsep{0cm}
    \item the signature of $\B$ at one real place of $F$ is $(k+1,k)$ if $k$ is even and $(k,k+1)$ if $k$ is odd, and
    \item $\B$ is positive definite at all other real places of $F$.
\end{itemize}
\end{proposition}

\begin{proof}
Margulis' Arithmeticity Theorem (see Theorem 16.3.1 in \cite{Morris_IntroductionArithmeticGroups}) implies that all lattices of $\SO(k+1,k)$ are arithmetic. From the classification of arithmetic subgroups of \cite{Morris_IntroductionArithmeticGroups} \S18.5, lattices of $\SO(k+1,k)$ are widely commensurable with $\SO(\B,\mathcal{O}_F)$ for $\mathcal{O}_F$ the ring of integers of a totally real number field $F$ and $\B\in\GL(n,F)$ a symmetric matrix with signature $(k+1,k)$ at one real place of $F$ and such that $\B$ is positive or negative definite over all other real places of $F$.

If $F=\mathds{Q}$, Meyer's Theorem (Corollary 2 in \S3 of Chapter 4 in \cite{Serre_CourseArithmetic}) implies that $\B$ is isotropic over $\mathds{Q}$. Thus $\SO(\B,\mathds{Z})$ is non-uniform. Hence $F\neq\mathds{Q}$.

Let $V\subset V_F$ be the set of real places of $F$ where $\B$ is negative definite. By Lemma \ref{lemmaelementnumberfield}, there exists $\lambda\in F$ such that for all $\iota\in V$ $\iota(\lambda)<0$ and for all $\iota\not\in V$ $\iota(\lambda)>0$. Then $\lambda \B$ is positive definite for all real places except one where it has signature $(k+1,k)$ and $\SO(\B,\mathcal{O}_F)=\SO(\lambda \B,\mathcal{O}_F)$. Up to replacing $\B$ with $\lambda \B$ we can assume that $V$ is empty.

Up to replacing $\B$ with $\Det(\B)\B$, we can assume that $\Det(\B)$ is a square. Then up to congruence, we can assume $\B$ has determinant 1.
\end{proof}

\begin{proposition}
\label{propclassificationlatticeSp}
Let $n\geq2$. The uniform lattices of $\Sp(2n,\mathds{R})$ are widely commensurable with $\SU(\I_n,\tc;\mathcal{O})$ for $\mathcal{O}$ an order of a quaternion algebra $A$ over a totally real number field $F\neq\mathds{Q}$ such that $A$ splits at exactly one real place of $F$.
\end{proposition}
\begin{proof}
Margulis' Arithmeticity Theorem (see Theorem 16.3.1. in \cite{Morris_IntroductionArithmeticGroups}) implies that all lattices of $\Sp(2n,\mathds{R})$ are arithmetic. From the classification of arithmetic subgroups of \cite{Morris_IntroductionArithmeticGroups} \S 18.5, uniform lattices of $\Sp(2n,\mathds{R})$ are widely commensurable with $\SU(\B,\tc;\mathcal{O})$, where $\mathcal{O}$ is an order of a quaternion division algebra $A$ over a totally real number field $F\neq\mathds{Q}$\footnote{Note that if $F=\mathds{Q}$ then the lattices in the classification are all non-uniform since $\tc$-Hermitian forms on a quaternion division algebra over $\mathds{Q}$ are always isotropic.}, $A$ ramifies at all $\iota\in V_F\setminus\{\sigma\}$ and $\B$ is a non-degenerate $\tc$-Hermitian matrix over $A$ such that for all $\iota\in V_F\setminus\{\sigma\}$ $\iota(\B)\in\GL(\mathcal{H})$ has signature $(n,0)$ or $(0,n)$. Denote by $$V_F^{+}=\{\iota\in V_F\ |\ \iota(\B)\ \textrm{has signature}\ (n,0)\}.$$ Let $\lambda\in F$ such that $\iota(\lambda)>0$ for all $\iota\in V_F^{+}$ and $\iota(\lambda)<0$ for all $\iota\not\in V_F^{+}$. It exists by Lemma \ref{lemmaelementnumberfield}. Then the $\tc$-Hermitian matrix
$\lambda\I_n$ has the same signature as $\B$ for all real embeddings of $F$, except maybe $\sigma$. As explained in \cite{Lewis_IsometryClassificationHermitianForms} \S 5, non-degenerate $\tc$-Hermitian forms on $A$ are classified by the set of signatures at all real embeddings of $F$ except $\sigma$. Hence $\SU(\B,\tc;\mathcal{O})$ is conjugate to $$\SU(\lambda\I_n,\tc;\mathcal{O})=\SU(\I_n,\tc;\mathcal{O}).$$\end{proof}

\section{Summary of results on forms}

Let $G$ be as in Table \ref{Table1}. To prove Theorem \ref{theoclassificationfuchsian}, we classify arithmetic subgroups of $G$ that contain the image of an arithmetic subgroup of $\SL(2,\mathds{R})$ under the irreducible representation $\tau_n$ (see  for the definition of $\tau_n$). Using Proposition \ref{propmorphismofalgebraicgroups}, it reduces to determine the $F$-algebraic groups with $\mathds{R}$-points $G$ that contain the image of an $F$-form of $\SL_2$ under $\tau_n$. Recall that forms of algebraic groups are classified using $1$-cocycle, see Theorem \ref{theoremcocycleandforms}.

This section is mainly a summary of results that were established in \cite{Audibert_Zariskidensesurfacegroups} in the case where $F=\mathds{Q}$ but for which the proof still works for any number field. Let $F$ be a totally real number field.
Let $$\xi:\GalF\to\PSL(2,\Qbar)=\Aut(\SL(2,\Qbar))$$ be a 1-cocycle.

\subsection{Compatible cocycles}

Let $F$ be a totally real number field.

\begin{definition}
Let $n\geq3$ and let $\G$ be an $F$-algebraic subgroup of $\SL_n$. We say that a 1-cocycle $\zeta:\GalF\to\Aut(\G(\Qbar))$ is \emph{$\tau_n$-compatible} with $\xi$ if
$$\tau_n(\prescript{}{\xi}{\SL_2}(F))<\prescript{}{\zeta}{\G}(F).$$
\end{definition}

Let $a,b\in\mathcal{O}_F$ and pick square roots $\sqrt{a},\sqrt{b}\in\Qbar$. Let $\T^{a,b}:\Gal(\Qbar/F)\to\PSL(2,\Qbar),\ \sigma\mapsto\T_{\sigma}^{a,b}$ with
\begin{equation*}
    \T_{\sigma}^{a,b}=\left\{\begin{array}{ll}
        \I_2 &\mbox{if $\sigma(\sqrt{a})=\sqrt{a}$ and $\sigma(\sqrt{b})=\sqrt{b}$}  \\
        \begin{pmatrix}
            1&0\\
            0&-1
        \end{pmatrix} &\mbox{if $\sigma(\sqrt{a})=\sqrt{a}$ and $\sigma(\sqrt{b})=-\sqrt{b}$}\\
        \begin{pmatrix}
        0&1\\
        1&0
        \end{pmatrix} &\mbox{if $\sigma(\sqrt{a})=-\sqrt{a}$ and $\sigma(\sqrt{b})=\sqrt{b}$}\\
        \begin{pmatrix}
        0&1\\
        -1&0
        \end{pmatrix} &\mbox{if $\sigma(\sqrt{a})=-\sqrt{a}$ and $\sigma(\sqrt{b})=-\sqrt{b}$.}
    \end{array}
    \right.
\end{equation*}We often write $\T_{\sigma}$ for $\T_{\sigma}^{a,b}$. The following lemma gives an explicit $1$-cocycle equivalent to $\xi$.

\begin{lemma}[Lemma 2.3 in \cite{Audibert_Zariskidensesurfacegroups}]
Let $\xi:\Gal(\Qbar/F)\to\PSL(2,\Qbar)$ be a $1$-cocycle. There exists $a,b\in\mathcal{O}_F$ and $\PP\in\SL(2,\Qbar)$ such that $$\xi(\sigma)=\PP^{-1}\T_{\sigma}^{a,b}\sigma(\PP)$$ for every $\sigma\in\Gal(\Qbar/F)$.
\end{lemma}

\begin{definition}
For all $n\geq3$ define $\J_n$ as
\begin{equation*}
    \begin{pNiceMatrix}
    &&&&(n-1)!\\
    &&&\Iddots&\\
    &&(-1)^{i-1}(n-i)!(i-1)!&&\\
    &\Iddots&&&\\
    (-1)^{n-1}(n-1)!&&&&&\\
\end{pNiceMatrix}.
\end{equation*}
\end{definition}

When $n$ is odd, it is a symmetric matrix of signature $(k+1,k)$ if $n=2k+1$ with $k$ even and $(k,k+1)$ if $n=2k+1$ with $k$ odd. When $n$ is even, $\J_n$ is antisymmetric. Note that it satisfies $\tau_n(\M)^\top\J_n\tau_n(\M)=\J_n$ for all $\M\in\SL(2,\mathds{C})$.
In the following proposition, we give an explicit description of the $1$-cocycles $\tau_n$-compatible with $\xi$. It will be used repeatedly in the next two sections.

\begin{proposition}[Proposition 2.6 in \cite{Audibert_Zariskidensesurfacegroups}]
\label{propcompatiblecocycles}
For any $n\geq3$, a 1-cocycle $\zeta:\GalF\to\Aut(\SL(n,\Qbar))$ is $\tau_n$-compatible with $\xi$ if and only if either
\begin{equation*}
    \zeta:\sigma\mapsto\Int(\tau_n(\PP^{-1}\T_{\sigma}^{a,b}\sigma(\PP)))
\end{equation*}
or there exists a quadratic field extension $F(\sqrt{d})$ of $F$ such that
\begin{equation*}
    \zeta:\sigma\mapsto\left\{\begin{array}{ll}
        \Int(\tau_n(\PP^{-1}\T_{\sigma}^{a,b}\sigma(\PP))) & \mbox{if $\sigma(\sqrt{d})=\sqrt{d}$} \\
        \Int(\tau_n(\PP^{-1}\T_{\sigma}^{a,b})\J_n^{-1})\circ\ \omega\circ\Int(\tau_n(\sigma(\PP))) & \mbox{if $\sigma(\sqrt{d})=-\sqrt{d}$}
    \end{array}\right.
\end{equation*}
where $\omega(\M)=(\M^\top)^{-1}$.
\end{proposition}

\subsection{Odd-dimensional cocycles and quadratic forms}
\label{subsectionquadraticform}

Let $n=2k+1$ with $k\geq1$. Let $\zeta:\Gal(\Qbar/F)\to\Aut(\SL(n,\Qbar))$ be the inner $1$-cocycle $\tau_n$-compatible with $\xi$. Proposition \ref{propcompatiblecocycles} implies that $\zeta$ has value in $\SO(\J_n,\Qbar)<\PSL(n,\Qbar)$. Proposition 2.8 in \cite{Platonov_AlgebraicgroupsNumbertheory} states that the first cohomology set $\HH^1(\Gal(\Qbar/F),\SO(\J_n,\Qbar))$ is in $1$-$1$ correspondence with the $F$-equivalence classes of quadratic forms of dimension $n$ over $F$ with discriminant equal to $\det(\J_n)$. We exhibit symmetric a matrix $\J_n^{a,b}$ which represents a quadratic form in the equivalence class associated to $\zeta$. It satisfies $$\prescript{}{\zeta}{\SO}(\J_n)(F)\simeq\SO(\J_n^{a,b},F).$$

The computations come from the proof of Lemma 3.3 in \cite{Audibert_Zariskidensesurfacegroups} and use the algorithm of Theorem \ref{Hilbert90} (Hilbert 90). They were done in the case $F=\mathds{Q}$ but hold for any number field. We recall here the results. If $a$ or $b$ is a square, let $\J^{a,b}_n=\J_n$. Suppose neither $a$ nor $b$ is a square. If $n\equiv1[4]$ let $\J_n^{a,b}$ be the diagonal matrix defined by
\begin{equation*}
    (\J_n^{a,b})_{ii}=\left\{\begin{array}{ll}
        -2a(n-i)!(i-1)! & \mbox{if $i\leq k$ and $i$ is odd,} \\
        -2b(n-i)!(i-1)! & \mbox{if $i\leq k$ and $i$ is even,} \\
        k!k! & \mbox{if $i=k+1$,} \\
        2ab(n-i)!(i-1)! & \mbox{if $i\geq k+2$ and $i$ is even,} \\
        2(n-i)!(i-1)! & \mbox{if $i\geq k+2$ and $i$ is odd.}
    \end{array}\right.
\end{equation*} If $n\equiv3[4]$ let $\J_n^{a,b}$ be the diagonal matrix defined by
\begin{equation*}
    (\J_n^{a,b})_{ii}=\left\{\begin{array}{ll}
        -2b(n-i)!(i-1)! & \mbox{if $i\leq k$ and $i$ is odd,}\\
        2(n-i)!(i-1)! & \mbox{if $i\leq k$ and $i$ is even,}\\
        -ak!k! & \mbox{if $i=k+1$,}\\
        -2a(n-i)!(i-1)! & \mbox{if $i\geq k+2$ and $i$ is even,}\\
        2ab(n-i)!(i-1)! & \mbox{if $i\geq k+2$ and $i$ is odd.}
    \end{array}\right.
\end{equation*}

\subsection{Inner cocycles and quaternion algebras}

For every $1$-cocycle $\zeta:\Gal(\Qbar/F)\to\PSL(n,\Qbar)$ define $$\prescript{}{\zeta}{\M_n(F)}=\{\M\in\M_n(\Qbar)\mid\zeta(\sigma)\sigma(\M)\zeta(\sigma)^{-1}=\M\ \forall\sigma\in\Gal(\Qbar/F)\}.$$ It satisfies $\prescript{}{\zeta}{\M_n(F)}\otimes_{F}\Qbar\simeq\M_n(\Qbar)$.

\begin{proposition}[Proposition 2.8 in \cite{Audibert_Zariskidensesurfacegroups}]
\label{propinnerformMn}
Let $\zeta:\Gal(\Qbar/F)\to\PSL(n,\Qbar)$ be a $1$-cocycle $\tau_n$-compatible with $\xi$. Then
\begin{equation*}
    \prescript{}{\zeta}{\M_n(F)}\simeq\left\{\begin{array}{ll}
        \M_n(F) & \mbox{if $n$ is odd} \\
        \M_{\frac{n}{2}}(\prescript{}{\xi}{\M_2(F)}) & \mbox{if $n$ is even.}
    \end{array}\right.
\end{equation*}
\end{proposition}

\section{Arithmetic subgroups in odd dimension}
\label{Sectionarithmeticodd}

In this section we prove Theorem \ref{theoclassificationfuchsian} for the case $G=\SO(k+1,k)$, $\textbf{G}_2(\mathds{R})$ and $\SL(2k+1,\mathds{R})$.

\subsection{Arithmetic subgroups of $\SO(k+1,k)$}
\label{subsectionlatticeSO}

Let $\mathcal{P}$ be a prime ideal of $\mathcal{O}_F$ and $a,b\in F^{\times}$. We denote by $(a,b)_{\mathcal{P}}$ the quaternion algebra $(a,b)_{F_{\mathcal{P}}}$ where $F_\mathcal{P}$ is the completion of $F$ at the place $\mathcal{P}$. We write $1$ for quaternion algebra $\M_2(F_{\mathcal{P}})$.

\begin{proposition}
\label{proplatticeSO}
Let $\Gamma$ be an arithmetic subgroup of $\SL(2,\mathds{R})$. Let $F$ be a totally real number field and $a,b\in F^{\times}$ such that $\Gamma$ is commensurable with the norm $1$ elements of an order of $(a,b)_F$.

Let $n=2k+1\geq3$ be odd. Then $\tau_n(\Gamma)$ lies in a subgroup of $\SO(\J_n,\mathds{R})$ widely commensurable with $\SO(\B,\mathcal{O}_F)$ for $B\in\SL(n,F)$ a symmetric matrix such that
\begin{itemize}
\setlength\itemsep{0cm}
    \item $\B$ has signature equal to $(k+1,k)$ is $k$ is even and $(k,k+1)$ if $k$ is odd at one real place of $F$,
    \item $\B$ is positive definite at all other real places of $F$ and
    \item $\B$ has Hasse invariant
\begin{equation*}
    \mathcal{E}_\mathcal{P}(\B)=\left\{\begin{array}{ll}
        1 & \mbox{if $n\equiv\pm1[8]$} \\
        (a,b)_{\mathcal{P}}\otimes(-1,-1)_{\mathcal{P}} & \mbox{if $n\equiv\pm3[8]$}
    \end{array}\right.
\end{equation*}
for every prime ideal $\mathcal{P}$ of $\mathcal{O}_F$.
\end{itemize}

Furthermore this is the only arithmetic subgroup of $\SO(\J_n,\mathds{R})$ that contain $\tau_n(\Gamma)$ up to commensurability.
\end{proposition}

\begin{proof}
Let $\xi:\GalF\to\PSL(2,\Qbar),\ \sigma\mapsto\T^{a,b}_{\sigma}.$
It is a $1$-cocycle such that $\Gamma$ is widely commensurable with $\prescript{}{\xi}{\SL_2}(\mathcal{O}_F)$. Define
\begin{equation*}
    \zeta:\GalF\to\SO(\J_n,\Qbar),\ \sigma\mapsto\tau_n(\T_\sigma^{a,b}).
\end{equation*}
It satisfies $\tau_n(\prescript{}{\xi}{\SL_2(F)}<\prescript{}{\zeta}{\SO(\J_n)(F)}$.
Proposition \ref{propmorphismofalgebraicgroups} shows that $\tau_n(\Gamma)$ lies in $\prescript{}{\zeta}{\SO(\J_n)}(\mathcal{O}_F)$ up to wide commensurability, which is widely commensurable with $\SO(\J_n^{a,b},\mathcal{O}_F)$, see \S\ref{subsectionquadraticform}.

If $n\equiv1[4]$ then its Hasse invariant at a prime ideal $\mathcal{P}$ of $\mathcal{O}_F$ is
\begin{equation*}
    (-1,-1)^{\frac{n-1}{4}}_{\mathcal{P}}\otimes\bigotimes_{j=1,3,\ldots,k-1}(a,bj(n-j))_{\mathcal{P}}\simeq(-1,-1)^{\frac{n-1}{4}}_{\mathcal{P}}\otimes(a,b^{\frac{n-1}{4}})_{\mathcal{P}},
\end{equation*}
since $\prod_{j=1,3,\ldots,k-1}j(n-j)$ is a square as can be shown by induction.

If $n\equiv3[4]$ then its Hasse invariant at a prime ideal $\mathcal{P}$ of $\mathcal{O}_F$ is
\begin{equation*}
    (-1,-1)_{\mathcal{P}}^{\frac{n+1}{4}}\otimes(a,2)_{\mathcal{P}}\otimes\bigotimes_{j=1,3,\ldots,k}(a,bj(n-j))_{\mathcal{P}}\simeq(-1,-1)_{\mathcal{P}}^{\frac{n+1}{4}}\otimes(a,b^{\frac{n+1}{4}})_{\mathcal{P}},
\end{equation*}
since $2\prod_{j=1,3,\ldots,k}j(n-j)$ is a square as can be shown by induction.

Conversely, suppose that an $\Lambda$ is an arithmetic subgroup of $\SO(\J_n,\mathds{R})$ that contains $\tau_n(\Gamma)$. Since $\SO(\J_n,\mathds{R})$ is simple, $\Lambda$ is widely commensurable with $\prescript{}{\eta}{\SO(\J_n)}(\mathcal{O}_L)$ for $L$ a number field and $\eta:\Gal(\Qbar/L)\to\SO(\J_n,\overline{\mathds{Q}})$ a $1$-cocycle. By Proposition \ref{propmorphismarithmeticlattices} we can assume that $L=F$.

We show that $\eta$ is $\tau_n$-compatible with $\xi$. For every $\sigma\in\Gal(\Qbar/F)$ denote by
$$\tau_n^{\sigma}:\SL_2(\Qbar)\to\SO(\J_n,\Qbar),\ g\mapsto\eta(\sigma)\circ\sigma\circ\tau_n\circ(\xi(\sigma)\circ\sigma)^{-1}(g).$$
This is an algebraic morphism that coincides with $\tau_n$ on a finite-index subgroup of $\Gamma$. Since any finite-index subgroup of $\Gamma$ is Zariski-dense in $\SL_2(\Qbar)$, $\tau_n=\tau_n^{\sigma}$. This means that $\tau_n(\prescript{}{\xi}{\SL_2}(F))<\prescript{}{\eta}{\SO(\J_n)}(F)$. Proposition \ref{propcompatiblecocycles} concludes the proof.
\end{proof}

\subsection{Arithmetic subgroups of $\textbf{G}_2(\mathds{R})$}

By $\textbf{G}_2(\mathds{R})$ we denote the connected centerless split real Lie group of type $\G_2$.

\begin{definition}
\label{definitionG2}
Let $R$ be a ring and let $a,b\in R$ be non-zero. Denote by $\times:R^7\times R^7\rightarrow R^7$
\begin{equation*}
    \begin{pmatrix}
    x_1\\
    x_2\\
    x_3\\
    x_4\\
    x_5\\
    x_6\\
    x_7
    \end{pmatrix},
    \begin{pmatrix}
    y_1\\
    y_2\\
    y_3\\
    y_4\\
    y_5\\
    y_6\\
    y_7
    \end{pmatrix}\mapsto
    \begin{pmatrix}
    6a(x_7y_4-x_4y_7)-4(x_2y_3-x_3y_2)-4a(x_6y_5-x_5y_6)\\
    24b(x_3y_1-x_1y_3)+24ab(x_7y_5-x_5y_7)-6a(x_6y_4-x_4y_6)\\
    60(x_2y_1-x_1y_2)+60a(x_7y_6-x_6y_7)-6a(x_5y_4-x_4y_5)\\
    240b(x_1y_7-x_7y_1)+40(x_2y_6-x_6y_2)-16b(x_3y_5-x_5y_3)\\
    60(x_1y_6-x_6y_1)-60(x_7y_2-x_2y_7)-6(x_3y_4-x_4y_3)\\
    24b(x_3y_7-x_7y_3)+24b(x_1y_5-x_5y_1)-6(x_2y_4-x_4y_2)\\
    6(x_1y_4-x_4y_1)-4(x_2y_5-x_5y_2)-4(x_6y_3-x_3y_6)
    \end{pmatrix}.
\end{equation*}
Define $\textbf{G}_2^{a,b}(R)=\{\M\in\SO(\J_7^{a,b},R)|\M(x\times y)=\M x\times\M y,\ \forall x,y\in R^7\}$.
\end{definition}

For any $a,b\in F^{\times}$ the $F$-algebraic group $\textbf{G}_2^{a,b}$ is simple of type $\G_2$.\footnote{Indeed it is conjugate over $\Qbar$ to $\textbf{G}_2(\mathds{R})$ as defined in Definition 3.6 in \cite{Audibert_Zariskidensesurfacegroups} using the matrix $\SSS$ defined in equation \eqref{equationmatrixSG2}.} For any $a,b\in\mathds{R}^{\times}$ which are both negative $\textbf{G}_2^{a,b}(\mathds{R})$ is isomorphic to the compact Lie group of type $\G_2$. If either $a$ or $b$ is positive, $\textbf{G}_2^{a,b}(\mathds{R})\simeq\textbf{G}_2(\mathds{R})$.

\begin{proposition}
\label{proplatticeG2}
Let $\Gamma$ be an arithmetic subgroup of $\SL(2,\mathds{R})$. Let $F$ be a totally real number field and $a,b\in F^{\times}$ such that $\Gamma$ is commensurable with the norm 1 elements of an order of $(a,b)_F$. Then $\tau_7(\Gamma)$ lies in a subgroup of \emph{$\textbf{G}_2(\mathds{R})$} widely commensurable with \emph{$\textbf{G}_2^{a,b}(\mathcal{O}_F)$}. Furthermore the latter is the only arithmetic subgroup of \emph{$\textbf{G}_2(\mathds{R})$} that contains $\tau_7(\Gamma)$ up to commensurability.
\end{proposition}

\begin{remark}
The subgroup $\textbf{G}_2^{a,b}(\mathcal{O}_F)$ is an arithmetic subgroup of $\textbf{G}_2(\mathds{R})$. Indeed, for all embeddings $\sigma:F\to\mathds{R}$ 
except one, $\textbf{G}_2^{\sigma(a),\sigma(b)}(\mathds{R})$ is compact. Moreover all arithmetic subgroups of $\textbf{G}_2(\mathds{R})$ are of this form, as we will see in Proposition \ref{propclassificationlatticeG2}.
\end{remark}

\begin{proof}
Let $\xi:\Gal(\overline{\mathds{Q}}/F)\to\PSL(2,\overline{\mathds{Q}}),\ \sigma\mapsto\T^{a,b}_{\sigma}$. It is a $1$-cocycle such that $\prescript{}{\xi}{\SL_2(\mathcal{O}_F)}$ is widely commensurable with $\Gamma$. Let $\zeta:\Gal(\overline{\mathds{Q}}/F)\to\textbf{G}_2(\overline{\mathds{Q}}),\ \sigma\mapsto\tau_7(\T^{a,b}_\sigma)$.
It satisfies $\tau_7(\prescript{}{\xi}{\SL_2(F)})<\prescript{}{\zeta}{\textbf{G}_2(F)}.$

We need to determine $\prescript{}{\zeta}{\textbf{G}_2(F)}$. Hilbert's 90, Theorem \ref{Hilbert90}, shows that there exists $\SSS\in\GL_7(\Qbar)$ such that for all $\sigma\in\Gal(\Qbar/F)$
$\zeta(\sigma)=\SSS^{-1}\sigma(\SSS).$ As computed in the proof of Lemma 3.3 in \cite{Audibert_Zariskidensesurfacegroups} we can take
\begin{equation}
\label{equationmatrixSG2}
\SSS=\frac{1}{2}\begin{pNiceMatrix}
\frac{1}{\sqrt{b}}&&&&&&-\frac{1}{\sqrt{b}}\\
&1&&&&-1&\\
&&\Ddots&&\Iddots&&\\
&&&\frac{2}{\sqrt{a}}&&&\\
&&\Iddots&&\Ddots&&\\
&\frac{1}{\sqrt{a}}&&&&\frac{1}{\sqrt{a}}&\\
\frac{1}{\sqrt{ab}}&&&&&&\frac{1}{\sqrt{ab}}
\end{pNiceMatrix}.
\end{equation}

Hence for $\M\in\prescript{}{\xi}{\SL_2(F)}$, $\SSS\tau_7(\M)\SSS^{-1}\in\SL(7,F)$. It appears that $\tau_7(\prescript{}{\xi}{\SL_2(F)})$ preserves the quadratic form $\J_7.$ Hence $\SSS\tau_7(\prescript{}{\xi}{\SL_2(F)})\SSS^{-1}$ preserves the quadratic form $\SSS^{-\top}\J_n\SSS^{-1}=\J_n^{a,b}.$ Also $\tau_7(\prescript{}{\xi}{\SL_2(F)})$ preserves the following cross product:
\begin{equation*}
    \begin{pmatrix}
    x_1\\
    x_2\\
    x_3\\
    x_4\\
    x_5\\
    x_6\\
    x_7
    \end{pmatrix},
    \begin{pmatrix}
    y_1\\
    y_2\\
    y_3\\
    y_4\\
    y_5\\
    y_6\\
    y_7
    \end{pmatrix}\mapsto
    \begin{pmatrix}
    6(x_1y_4-x_4y_1)-4(x_2y_3-x_3y_2)\\
    24(x_1y_5-x_5y_1)-6(x_2y_4-x_4y_2)\\
    60(x_1y_6-x_6y_1)-6(x_3y_4-x_4y_3)\\
    120(x_1y_7-x_7y_1)+20(x_2y_6-x_6y_2)-8(x_3y_5-x_5y_3)\\
    60(x_2y_7-x_7y_2)-6(x_4y_5-x_5y_4)\\
    24(x_3y_7-x_7y_3)-6(x_4y_6-x_6y_4)\\
    6(x_4y_7-x_7y_4)-4(x_5y_6-x_6y_5)
    \end{pmatrix}.
\end{equation*} Indeed, one can check that $\tau_7(\SL(2,\mathds{Z}))$ preserves this cross product and use the fact that $\SL(2,\mathds{Z})$ is Zariski-dense in $\SL(2,\Qbar)$. It implies that $\SSS\tau_7(\prescript{}{\xi}{\SL_2(F)})\SSS^{-1}$ preserves $\times$ as defined in Definition \ref{definitionG2}. It follows that
$$\SSS\tau_7(\prescript{}{\xi}{\SL_2(F)})\SSS^{-1}<\textbf{G}_2^{a,b}(F).$$
Since there is a unique $1$-cocycle $\tau_7$-compatible with $\xi$ that has values in $\textbf{G}_2(\Qbar)$, see Proposition \ref{propcompatiblecocycles}, there is only one $F$-form of $\textbf{G}_2(\Qbar)$ that contains $\tau_7(\prescript{}{\xi}{\SL_2(F)})$. We deduce that $\prescript{}{\zeta}{\textbf{G}_2(F)}\simeq\textbf{G}_2^{a,b}(F)$. Proposition \ref{propmorphismofalgebraicgroups} shows that $\tau_7(\Gamma)$ is contained in $\prescript{}{\zeta}{\textbf{G}_2(\mathcal{O}_F)}\simeq\textbf{G}_2^{a,b}(\mathcal{O}_F)$ up to finite index.

The converse statement is proven as in the proof of Proposition \ref{proplatticeSO}.
\end{proof}

Finally, we conclude this part by the classification of lattices of $\textbf{G}_2(\mathds{R})$ as we will need it for the proof of Theorem \ref{theoclassificationfuchsian}. It is surely widely known.

\begin{proposition}
\label{propclassificationlatticeG2}
All lattices of $\emph{\textbf{G}}_2(\mathds{R})$ are widely commensurable with $\emph{\textbf{G}}^{a,b}_2(\mathcal{O}_F)$ where $F$ is a totally real number field and $a,b\in F^{a,b}$ are such that for all embeddings $\sigma:F\to\mathds{R}$ except one $\sigma(a)$ and $\sigma(b)$ are negative.
\end{proposition}

\begin{proof}
Let $\Gamma<\textbf{G}_2(\mathds{R})$ be a lattice. By Margulis' Arithmeticity Theorem (Theorem 16.3.1 in \cite{Morris_IntroductionArithmeticGroups}) $\Gamma$ is arithmetic. By Proposition \ref{propGsimplearithmetic}, there exists a semisimple $F$-algebraic group $\HH$ such that $\HH(F_v)$ is compact for all archimedean places $v$ except one, denoted $v_0$, where we have an isogeny $$\phi:\HH(F_{v_0})\to \textbf{G}_2(\mathds{R})$$ satisfying that $\phi(\HH(\mathcal{O}_F))$ is commensurable with $\Gamma$.\footnote{Corollary 5.5.15 in \cite{Morris_IntroductionArithmeticGroups} allows us to assume that $\HH(\mathds{R})$ is connected.} The group $\HH(F_{v_0})$ cannot be the universal cover of $\textbf{G}_2(\mathds{R})$ since the latter is not algebraic. Hence $\phi$ is an isomorphism and $\HH$ is an $\mathds{R}/F$-form of $\textbf{G}_2$. Those are classified by $\mathds{R}/F$-forms of the split octonion algebra over $\mathds{R}$ which are determined by their norm, see Theorem 1.7.1 in \cite{Springer_OctonionJordanAlgebrasExceptionalGroups}. The norm associated to $\HH$ is positive definite over all real places of $F$ except one. So up to scalar multiplication, it is equivalent to $\J_7^{a,b}$ with $a,b\in F^{\times}$ which are both negative at all real places of $F$ except one. Thus the octonion algebras associated to $\HH$ and $\textbf{G}_2^{a,b}$ are isomorphic and finally $\HH\simeq\textbf{G}_2^{a,b}$.
\end{proof}

\subsection{Arithmetic subgroups of $\SL(2k+1,\mathds{R})$}

\begin{proposition}
\label{proplatticeSLodd}
Let $\Gamma$ be an arithmetic subgroup of $\SL(2,\mathds{R})$. Let $F$ be a totally real number field and $A$ a quaternion algebra over $F$ such that $\Gamma$ is commensurable with the norm 1 elements of an order of $A$. Suppose that $F\neq\mathds{Q}$.\footnote{For the corresponding statement when $F=\mathds{Q}$ see Proposition A.1 in \cite{Audibert_Zariskidensesurfacegroups}.}

Let $n\geq3$ be odd. Then $\tau_n(\Gamma)$ lies in a subgroup of $\SL(n,\mathds{R})$ widely commensurable with $\SU(\I_n,\sigma;\mathcal{O}_F[\sqrt{d}])$ for every $d\in\mathcal{O}_F$ which is positive at exactly the same real places of $F$ where $A$ splits and $\sigma\in\Gal(F(\sqrt{d})/F)$ non-trivial.

Furthermore these are the only arithmetic subgroups of $\SL(n,\mathds{R})$ that contain $\tau_n(\Gamma)$ up to commensurability.
\end{proposition}

\begin{proof}
Let $a,b\in F^{\times}$ such that $A\simeq(a,b)_F$. Let $\xi:\Gal(\overline{\mathds{Q}}/F)\to\PSL(2,\overline{\mathds{Q}})$, $\sigma\mapsto\T^{a,b}_{\sigma}$. It is a $1$-cocycle such that $\prescript{}{\xi}{\SL_2(\mathcal{O}_F)}$ is widely commensurable with $\Gamma$.
We denote $\T_{\sigma}$ for $\T_{\sigma}^{a,b}$. Let $\zeta:\Gal(\overline{\mathds{Q}}/F)\to\Aut(\SL_n(\overline{\mathds{Q}}))$ be a $\tau_n$-compatible $1$-cocycle.

If $\zeta$ is inner then Proposition \ref{propinnerformMn} shows that $\prescript{}{\zeta}{\SL_n}(F)\simeq\SL(n,F).$ However, $\SL(n,\mathcal{O}_F)$ is not a lattice of $\SL(n,\mathds{R})$ if $F\neq\mathds{Q}$.

Suppose from now on that $\zeta$ is not inner. Denote by $F(\sqrt{d})$ the associated quadratic extension of $F$. Recall that $\M\in\prescript{}{\zeta}{\SL_n(F)}$ if and only if
\begin{equation*}
    \tau_n(\T_{\sigma})\sigma(\M)\tau_n(\T_{\sigma})^{-1}=\M
\end{equation*}
for all $\sigma\in\Gal(\Qbar/F)$ such that $\sigma(\sqrt{d})=\sqrt{d}$ and
\begin{equation*}
    \tau_n(\T_{\sigma})\J_n^{-1}\sigma(\M)^{-\top}\J_n\tau_n(\T_{\sigma})^{-1}=\M
\end{equation*}
for all $\sigma\in\Gal(\Qbar/F)$ such that $\sigma(\sqrt{d})=-\sqrt{d}$, see Proposition \ref{propcompatiblecocycles}. From Hilbert's 90, Theorem \ref{Hilbert90}, there exists $\SSS\in\GL(n,\Qbar)$ such that $\tau_n(\T_{\sigma})=\SSS^{-1}\sigma(\SSS)$. The first set of equations is equivalent to $$\SSS\M\SSS^{-1}\in\SL(n,F(\sqrt{d})).$$ Secondly, for all $\sigma\in\Gal(\Qbar/F)$ such that $\sigma(\sqrt{d})=-\sqrt{d}$
\begin{align*}
    &\tau_n(\T_{\sigma})\J_n^{-1}\sigma(\M)^{\top}\J_n\tau_n(\T_{\sigma})^{-1}=\M^{-1}\\
    \Leftrightarrow&\ \tau_n(\T_{\sigma})^{-\top}\sigma(\M)^{\top}\tau_n(\T_{\sigma})^{\top}\J_n\M=\J_n
\end{align*}
since $\tau_n(\T_{\sigma})$ and $\J_n$ commute and since $\tau_n(\T_{\sigma})^{-\top}=\tau_n(\T_{\sigma})$,
\begin{align*}
    \Leftrightarrow\ \sigma(\SSS\M\SSS^{-1})^{\top}\SSS^{-\top}\J_n\SSS^{-1}(\SSS\M\SSS^{-1})=\SSS^{-\top}\J_n\SSS^{-1}.
\end{align*}
As explained in \S\ref{subsectionquadraticform}, we can take $\SSS$ such that $\SSS^{-\top}\J_n\SSS^{-1}=\J_n^{a,b}$. Thus $$\SSS\prescript{}{\zeta}{\SL_n(F)}\SSS^{-1}=\SU(\J_n^{a,b},\sigma;F(\sqrt{d}))$$with $\sigma\in\Gal(F(\sqrt{d})/F)$ non-trivial. From \S 4 of \cite{Lewis_IsometryClassificationHermitianForms}, we see that $\sigma$-Hermitian forms over $F(\sqrt{d})$ are classified up to equivalence by their rank, their discriminant and their signatures at real places of $F$ where $d$ is negative. Hence $\J_n^{a,b}$ is equivalent to $\I_n$. Proposition \ref{propmorphismofalgebraicgroups} shows that $\tau_n(\Gamma)$ is contained in $\prescript{}{\zeta}{\SL_n(\mathcal{O}_F)}$ is widely commensurable with $\SU(\I_n,\sigma;\mathcal{O}_F[\sqrt{d}])$.

The converse statement is proven as in the proof of Proposition \ref{proplatticeSO}.
\end{proof}

\section{Arithmetic subgroups in even dimension}
\label{Sectionarithmeticsubgroupseven}

\subsection{Standard cocycle}
\label{subsectionstandardcocycle}

Let $a,b\in F^{\times}$ and $$\zeta:\Gal(\Qbar/F)\to\PSL(2n,\Qbar),\ \sigma\mapsto\tau_{2n}(\T^{a,b}_{\sigma}).$$ We will show that $\zeta$ is equivalent to the following ``standard'' $1$-cocycle.

\begin{definition}
\label{defstandardcocycle}
The $1$-cocycle defined by
\begin{equation*}
    \eta:\Gal(\Qbar/F)\to\PSL(2n,\Qbar)
\end{equation*}
\begin{equation*}
    \sigma\mapsto\left\{\begin{array}{ll}
        \Diag\left(\begin{pmatrix}1&0\\0&1\end{pmatrix}\right) & \mbox{if $\sigma(\sqrt{a})=\sqrt{a}$ and $\sigma(\sqrt{b})=\sqrt{b}$} \\
        \Diag\left(\begin{pmatrix}
        1&0\\
        0&-1
        \end{pmatrix}\right) & \mbox{if $\sigma(\sqrt{a})=\sqrt{a}$ and $\sigma(\sqrt{b})=-\sqrt{b}$}\\
        \Diag\left(\begin{pmatrix}0&1\\
        1&0\end{pmatrix}\right) & \mbox{if $\sigma(\sqrt{a})=-\sqrt{a}$ and $\sigma(\sqrt{b})=\sqrt{b}$}\\
        \Diag\left(\begin{pmatrix}0&1\\-1&0\end{pmatrix}\right) & \mbox{if $\sigma(\sqrt{a})=-\sqrt{a}$ and $\sigma(\sqrt{b})=-\sqrt{b}$.}
    \end{array}\right.
\end{equation*} is called the \emph{standard 1-cocycle associated to \emph{(}a,b\emph{)}}. Here $\Diag(\M)$ is block diagonal matrix with each block on the diagonal being equal to $\M$.
\end{definition}

Computations show that $\prescript{}{\eta}{\M_{2n}(F)}$
are $2$-by-$2$ block matrices with each block of the form
\begin{equation*}
   \begin{pmatrix}
    x_0+\sqrt{a}x_1&\sqrt{b}x_2+\sqrt{ab}x_3\\
    \sqrt{b}x_2-\sqrt{ab}x_3&x_0-\sqrt{a}x_1
    \end{pmatrix}
\end{equation*} for some $x_i\in F$.

The $1$-cocycle $\zeta$ does not lift to $\SL(2n,\Qbar)$, preventing us to use Hilbert's 90, Theorem \ref{Hilbert90}. To resolve this issue, we consider the following $1$-cocycle.

\begin{lemma}
The map $\zeta\eta^{-1}:\Gal(\Qbar/F)\to\prescript{}{\eta}{\PSL_n(\Qbar)}$, $\sigma\mapsto\zeta(\sigma)\eta(\sigma)^{-1}$ is a $1$-cocycle.
\end{lemma}

Note that in $\prescript{}{\eta}{\PSL_n(\Qbar)}$ the action of the Galois group is twisted by $\eta$, i.e. for $\sigma\in\Gal(\Qbar/F)$ and $\M\in\prescript{}{\eta}{\PSL_n(\Qbar)}$ we have $\sigma\cdot\M=\eta(\sigma)\sigma(\M)\eta(\sigma)^{-1}$.

\begin{proof}
For all $\sigma,\tau\Gal(\Qbar/F)$
\begin{align*}
    \zeta\eta^{-1}(\sigma\tau)&=\zeta(\sigma)\sigma(\zeta(\tau))\sigma(\eta(\tau))^{-1}\eta(\sigma)^{-1}\\
    &=\zeta(\sigma)\eta(\sigma)^{-1}(\eta(\sigma)\sigma(\zeta(\tau)\eta(\tau)^{-1})\eta(\sigma)^{-1})\\
    &=\zeta\eta^{-1}(\sigma)\sigma\cdot\zeta\eta^{-1}(\tau).
\end{align*}
\end{proof}

The $1$-cocycle $\zeta\eta^{-1}$ lifts to a $1$-cocycle $\chi:\Gal(\Qbar/F)\to\prescript{}{\eta}{\SL_{2n}(\Qbar)}$
\begin{equation*}
    \sigma\mapsto\left\{\begin{array}{ll}
        \I_{2n} & \mbox{if $\sigma(\sqrt{a})=\sqrt{a}$} \\
        \begin{pNiceMatrix}
        &&&1&0\\
        &&&0&1\\
        &&\Iddots&&\\
        1&0&&&\\
        0&1&&&
        \end{pNiceMatrix} & \mbox{if $\sigma(\sqrt{a})=-\sqrt{a}$.}
    \end{array}\right.
\end{equation*}
Hilbert's 90, Theorem \ref{Hilbert90}, tells us that there exists $\PP\in\GL(n,\Qbar)$ such that for all $\sigma$, $\zeta\eta^{-1}(\sigma)=\PP^{-1}\eta(\sigma)\sigma(\PP)\eta(\sigma)^{-1}$ and its proof gives an algorithm to determine such a matrix $\PP$. The algorithm is described in \S 3.2 of \cite{Audibert_Zariskidensesurfacegroups}. We give here the conclusions. If $a$ is a square, let $\PP=\I_{2n}$. Assume $a$ is not a square. If $n$ is even let
\begin{equation*}
    \PP=\frac{1}{2}\begin{pNiceMatrix}
    \I_2&&&&&\I_2\\
    &\Ddots&&&\Iddots&\\
    &&\I_2&\I_2&&\\
    &&\frac{1}{\sqrt{a}}\I_2&\text{-}\frac{1}{\sqrt{a}}\I_2&&\\
    &\Iddots&&&\Ddots&\\
    \frac{1}{\sqrt{a}}\I_2&&&&&\text{-}\frac{1}{\sqrt{a}}\I_2
    \end{pNiceMatrix}.
\end{equation*} If $n$ is odd let
\begin{equation*}
    \PP=\frac{1}{2}\begin{pNiceMatrix}
    \I_2&&&&&&\I_2\\
    &\Ddots&&&&\Iddots&\\
    &&\I_2&&\I_2&&\\
    &&&2\I_2&&&\\
    &&\frac{1}{\sqrt{a}}\I_2&&\text{-}\frac{1}{\sqrt{a}}\I_2&&\\
    &\Iddots&&&&\Ddots&\\
    \frac{1}{\sqrt{a}}\I_2&&&&&&\text{-}\frac{1}{\sqrt{a}}\I_2
    \end{pNiceMatrix}.
\end{equation*}
For all $\sigma\in\Gal(\Qbar/F)$, $\zeta(\sigma)=\PP^{-1}\eta(\sigma)\sigma(\PP).$ In particular $\zeta$ and $\eta$ are equivalent.

\subsection{Arithmetic subgroups of $\Sp(2n,\mathds{R})$}

\begin{proposition}
\label{proplatticeSp}
Let $\Gamma$ be an arithmetic subgroup of $\SL(2,\mathds{R})$. Let $\mathcal{O}$ be an order of a quaternion algebra over a totally real number field such that $\Gamma$ is widely commensurable with $\mathcal{O}^1$. Let $n\geq2$. Then $\tau_{2n}(\Gamma)$ lies in a subgroup widely commensurable with $\SU(\I_n,\tc;\mathcal{O})$. Furthermore, this is the only lattice of $\Sp(2n,\mathds{R})$ that contains $\tau_{2n}(\Gamma)$ up to commensurability.
\end{proposition}

\begin{proof}
Let $F$ be a totally real number field and $a,b\in F^{\times}$ such that $\mathcal{O}$ is an order of $(a,b)_F$. We can assume that $a,b\in\mathcal{O}_F$. Let $\xi:\Gal(\Qbar/F)\to\PSL(2,\Qbar)$, $\sigma\mapsto\T^{a,b}_\sigma$. It is a $1$-cocycle such that $\Gamma$ is commensurable with $\prescript{}{\xi}{\SL_{2}(\mathcal{O}_F)}$. Let $$\zeta:\Gal(\Qbar/F)\to\PSp_{2n}(\J_{2n},\Qbar)$$ be the $\tau_{2n}$-compatible $1$-cocycle. We want to determine $\prescript{}{\zeta}{\Sp_{2n}(\J_{2n})}(F)$.

Denote $\eta$ the standard $1$-cocycle associated to $(a,b)$ (see Definition \ref{defstandardcocycle}). Recall from \S\ref{subsectionstandardcocycle}, $\zeta(\sigma)=\PP^{-1}\eta(\sigma)\sigma(\PP)$ for all $\sigma\in\Gal(\Qbar/F)$ with $\PP$ defined in \S\ref{subsectionstandardcocycle}. Then $\prescript{}{\zeta}{\Sp_{2n}(\J_{2n})(F)}$
\begin{align*}
    &=\{\M\in\Sp_{2n}(\J_{2n})(\Qbar)\ |\ \zeta(\sigma)\sigma(\M)\zeta(\sigma)^{-1}=\M\ \forall\sigma\in\Gal(\Qbar/F)\}\\
    &=\{\M\in\Sp_{2n}(\J_{2n})(\Qbar)\ |\ \eta(\sigma)\sigma(\PP\M\PP^{-1})\eta(\sigma)^{-1}=\PP\M\PP^{-1}\ \forall\sigma\in\Gal(\Qbar/F)\}\\
    &=\{\X\in\prescript{}{\eta}{\SL_{2n}(F)}\ |\ \X^\top\PP^{-\top}\J_{2n}\PP^{-1}\X=\PP^{-\top}\J_{2n}\PP^{-1}\}.
\end{align*}Denote by $A$ the quaternion algebra
\begin{equation*}
    \begin{Bmatrix}
    \begin{pmatrix}
    x_0+\sqrt{a}x_1&\sqrt{b}x_2+\sqrt{ab}x_3\\
    \sqrt{b}x_2-\sqrt{ab}x_3&x_0-\sqrt{a}x_1
    \end{pmatrix}\ \Bigg|\ x_i\in F
    \end{Bmatrix}.
\end{equation*} It is isomorphic to $(a,b)_F$. Further more $\prescript{}{\eta}{\SL_{2n}(F)}=\SL(n,A)$. Denote by $\overline{\phantom{s}}$ the conjugation of $A$ and let 
$$\K=\Diag\left(\begin{pmatrix}
0&1\\
-1&0
\end{pmatrix}\right).$$
Then for all $\M\in\prescript{}{\eta}{\SL_{2n}(F)}$,
$\overline{\M}^t=(\K\M\K^{-1})^\top$ where $\M^t$ is the matrix obtained by transposing $\M$ as an element of $\SL(n,A)$. We emphasize only the position of its $2$-by-$2$ blocks changes, those blocks are not themselves transposed. Hence
\begin{align*}
    \prescript{}{\zeta}{\Sp_{2n}(\J_{2n})(F)}=\{\M\in\prescript{}{\eta}{\SL_{2n}(F)}\ |\ \overline{\M}^t\J_{2n}^*\M=\J_{2n}^*\}.
\end{align*} where $\J^*_{2n}=\K\PP^{-\top}\J_{2n}\PP^{-1}$. Computations show that $\J_{2n}^*$ is a $\overline{\phantom{s}}$-Hermitian matrix. We now show that $\J_{2n}^*$ is equivalent as a $\overline{\phantom{s}}$-Hermitian matrix to $-\I_{2n}$. If $n$ is even, let
\begin{equation*}
    \N=\begin{pNiceMatrix}
    \I_2&&&&&\I_2\\
    &\Ddots&&&\Iddots&\\
    &&\I_2&\I_2&&\\
    &&\text{-}\D&\D&&\\
    &\Iddots&&&\Ddots&\\
    \text{-}\D&&&&&\D
    \end{pNiceMatrix}
\end{equation*} with $$\D=\begin{pmatrix}
\frac{1}{\sqrt{a}}&0\\
0&-\frac{1}{\sqrt{a}}
\end{pmatrix}.$$ If $n$ is odd, let
\begin{equation*}
    \N=\begin{pNiceMatrix}
    \I_2&&&&&&\I_2\\
    &\Ddots&&&&\Iddots&\\
    &&\I_2&&\I_2&&\\
    &&&2\I_2&&&\\
    &&\text{-}\D&&\D&&\\
    &\Iddots&&&&\Ddots&\\
    \text{-}\D&&&&&&\D
    \end{pNiceMatrix}.
\end{equation*}

Then $\overline{\N}^t\J_{2n}^*\N$ is the diagonal matrix defined by
\begin{equation*}
    (\overline{\N}^t\J_{2n}^*\N)_{ii}=\left\{\begin{array}{ll}
        -4(2n-i-1)!i! & \mbox{if $i$ is odd}\\
        -4(2n-i)!(i-1)! & \mbox{if $i$ is even.}
    \end{array}\right.
\end{equation*}
Since $\Gamma$ is an arithmetic subgroup of $\SL(2,\mathds{R})$, $F$ is totally real and $A$ ramifies at all embeddings of $F$ except one, which we denote by $\iota$. As we can see in \S5 of \cite{Lewis_IsometryClassificationHermitianForms}, non-degenerate $\tc$-Hermitian forms on $A$ are classified by their signatures at all real embeddings of $F$ except $\iota$. Here the set of signatures is always $(0,2n)$. Hence $-\J_{2n}^*$ is equivalent as a $\tc$-Hermitian form to $\I_{2n}$. Finally $$\prescript{}{\zeta}{\Sp_{2n}(\J_{2n})(F)}\simeq\SU(\I_n,\tc;A).\footnote{Here $\I_n$ is viewed as a matrix with entries in $A$. Hence it is the matrix $\I_{2n}$ when we see elements of $A$ as $2$-by-$2$ matrices.}$$
Proposition \ref{propmorphismofalgebraicgroups} shows that $\tau_{2n}(\Gamma)$ is contained in $\prescript{}{\zeta}{\Sp_{2n}(\J_{2n})}(\mathcal{O}_F)$ which is widely commensurable with $\SU(\I_n,\overline{\phantom{s}};\mathcal{O})$.

The converse statement is proven as in the proof of Proposition \ref{proplatticeSO}.
\end{proof}

\subsection{Arithmetic subgroups of $\SL(2n,\mathds{R})$}

\begin{proposition}
\label{proplatticeSLeven}
Let $\Gamma$ be an arithmetic subgroup of $\SL(2,\mathds{R})$. Let $\mathcal{O}$ be an order of a quaternion algebra $A$ over a totally real number field $F$ such that $\Gamma$ is widely commensurable with $\mathcal{O}^1$. Suppose that $F\neq\mathds{Q}$.\footnote{See Proposition A.2 in \cite{Audibert_Zariskidensesurfacegroups} for the case $F=\mathds{Q}$.}

Let $n\geq2$. Then $\tau_{2n}(\Gamma)$ lies in a subgroup of $\SL(2n,\mathds{R})$ widely commensurable with $\SU(\I_n,\tc\otimes\sigma;\mathcal{O}\otimes\mathcal{O}_F[\sqrt{d}])$ for every $d\in\mathcal{O}_F$ which is positive exactly at the real place where $A$ splits and $\sigma\in\Gal(F(\sqrt{d})/F)$ non-trivial.

Furthermore, those are the only lattices of $\SL(2n,\mathds{R})$ that contain $\tau_{2n}(\Gamma)$ up to commensurability.
\end{proposition}

\begin{proof}
Let $a,b\in\mathcal{O}_F$ such that $A\simeq(a,b)_F$. Let $\xi:\Gal(\Qbar/F)\to\PSL(2,\Qbar)$, $\sigma\mapsto\T^{a,b}_{\sigma}$. It is a $1$-cocycle such that $\Gamma$ is commensurable with $\prescript{}{\xi}{\SL_2(\mathcal{O}_F)}$. Let $\zeta:\Gal(\Qbar/F)\to\Aut(\SL_{2n}(\Qbar))$ be a $\tau_{2n}$-compatible $1$-cocycle. We want to determine $\prescript{}{\zeta}{\SL_{2n}(F)}$.

If $\zeta$ is inner, Proposition \ref{propinnerformMn} shows that $\prescript{}{\zeta}{\SL_{2n}(F)}\simeq\SL(n,A)$. However, $\SL(n,\mathcal{O})$ is not a lattice of $\SL(2n,\mathds{R})$ since $F\neq\mathds{Q}$.

Suppose that $\zeta$ is not inner. Denote by $F(\sqrt{d})$ the corresponding quadratic extension (see Proposition \ref{propcompatiblecocycles}). We can suppose that $d\in\mathcal{O}_F$. Let $\sqrt{d}$ be a square root of $d$.

Recall that for any $\M\in\SL(2n,\Qbar)$, $\M\in\prescript{}{\zeta}{\SL_{2n}(F)}$ if and only if
$$\tau_{2n}(\T_{\sigma})\sigma(\M)\tau_{2n}(\T_{\sigma})^{-1}=\M$$
for all $\sigma\in\Gal(\Qbar/F)$ such that $\sigma(\sqrt{d})=\sqrt{d}$ and
$$\tau_{2n}(\T_{\sigma})\J_{2n}^{-1}\sigma(\M)^{-\top}\J_{2n}\tau_{2n}(\T_{\sigma})^{-1}=\M$$ for all $\sigma\in\Gal(\Qbar/F)$ such that $\sigma(\sqrt{d})=-\sqrt{d}$. In \S4.1 we proved that $\tau_{2n}(\T_{\sigma})=\PP^{-1}\eta(\sigma)\sigma(\PP)$. Hence $\M\in\prescript{}{\zeta}{\SL_{2n}(F)}$ if and only if $\PP\M\PP^{-1}\in\prescript{}{\eta}{\SL_{2n}(F(\sqrt{d}))}\simeq\SL(n,A\otimes_{F}F(\sqrt{d}))$ and for all $\sigma\in\Gal(\Qbar/F)$ such that $\sigma(\sqrt{d})=-\sqrt{d}$
\begin{align*}
    \sigma(\M)^\top\J_{2n}\sigma(\PP)^{-1}\eta(\sigma)^{-1}\PP\M&=\J_{2n}\sigma(\PP)^{-1}\eta(\sigma)^{-1}\PP\\
    \Leftrightarrow\ \sigma(\PP\M\PP^{-1})^\top\sigma(\PP)^{-\top}\J_{2n}(\eta(\sigma)\sigma(\PP))^{-1}\PP\M\PP^{-1}&=\sigma(\PP)^{-\top}\J_{2n}(\eta(\sigma)\sigma(\PP))^{-1}\\
    \Leftrightarrow\ \sigma(\PP\M\PP^{-1})^\top\eta(\sigma)\PP^{-\top}\J_{2n}\PP^{-1}(\PP\M\PP^{-1})&=\eta(\sigma)\PP^{-\top}\J_{2n}\PP^{-1}
\end{align*}
by applying $\sigma(.)^\top$ to each side of the equation since $\sigma^2(\PP\M\PP^{-1})=\PP\M\PP^{-1}$. We will show that this set of equations is equivalent to the defining equation of a group conjugated to $\SU(\I_n,\overline{\phantom{s}}\otimes\sigma;A\otimes_{F}F(\sqrt{d}))$. Denote by 
$$\K=\Diag\left(\begin{pmatrix}
0&1\\
-1&0
\end{pmatrix}\right).$$ The quaternion algebra $A$ is isomorphic to
\begin{equation*}
    \left\{\begin{pmatrix}
    x_0+\sqrt{a}x_1&\sqrt{b}x_2+\sqrt{ab}x_3\\
    \sqrt{b}x_2-\sqrt{ab}x_3&x_0-\sqrt{a}x_1
    \end{pmatrix}\ \Bigg|\ x_i\in F\right\}.
\end{equation*} We use this isomorphism to embed $A$ in $\M_2(\Qbar)$. Then for all $\X\in\prescript{}{\eta}{\SL_{2n}(F(\sqrt{d}))}$,
$\overline{\X}^t=(\K\X\K^{-1})^\top$ where $\X^t$ is the matrix obtained by transposing $\X$ as an element of $\SL(n,A\otimes_{F}F(\sqrt{d}))$. For any $\sigma\in\Gal(\Qbar/F)$ denote by
\begin{align*}\sigma^*:\prescript{}{\eta}{\SL_{2n}(F(\sqrt{d}))}&\to\prescript{}{\eta}{\SL_{2n}(F(\sqrt{d}))}\\
\end{align*} the map defined on $2$-by-$2$ blocks by
\begin{equation*}
    \begin{psmallmatrix}
    x_0+\sqrt{a}x_1&\sqrt{b}x_2+\sqrt{ab}x_3\\
    \sqrt{b}x_2-\sqrt{ab}x_3&x_0-\sqrt{a}x_1
    \end{psmallmatrix}\mapsto\begin{psmallmatrix}
    \sigma(x_0)+\sqrt{a}\sigma(x_1)&\sqrt{b}\sigma(x_2)+\sqrt{ab}\sigma(x_3)\\
    \sqrt{b}\sigma(x_2)-\sqrt{ab}\sigma(x_3)&\sigma(x_0)-\sqrt{a}\sigma(x_1)
    \end{psmallmatrix}.
\end{equation*}
Fix $\sigma\in\Gal(\Qbar/F)$ such that $\sigma(\sqrt{d})=-\sqrt{d}$. Let $\tau\in\Gal(\Qbar/F)$ such that $\tau(\sqrt{d})=-\sqrt{d}$ and denote by $s=\sigma^{-1}\tau$. Then for all $\X\in\prescript{}{\eta}{\SL_{2n}(F(\sqrt{d}))}$
\begin{align*}
    &\tau(\X)^\top\eta(\tau)\PP^{-\top}\J_{2n}\PP^{-1}\X=\eta(\tau)\PP^{-\top}\J_{2n}\PP^{-1}\\
    \Leftrightarrow\ &\eta(\sigma s)^{-1}\sigma s(\X)^\top\eta(\sigma s)\PP^{-\top}\J_{2n}\PP^{-1}\X=\PP^{-\top}\J_{2n}\PP^{-1}\\
    \Leftrightarrow\ & \eta(\sigma)^{-1}\sigma(\eta(s)^{-1}s(\X)^\top\eta(s))\eta(\sigma)\PP^{-\top}\J_{2n}\PP^{-1}\X=\PP^{-\top}\J_{2n}\PP^{-1},\\
    \textrm{since}\  &\eta(\sigma)\eta(s)=\eta(s)\eta(\sigma)\in\PSL(2n,\Qbar)\\
    \Leftrightarrow\ & \eta(\sigma)^{-1}\sigma(\eta(s)s(\X)\eta(s)^{-1})^\top\eta(\sigma)\PP^{-\top}\J_{2n}\PP^{-1}\X=\PP^{-\top}\J_{2n}\PP^{-1}\\
    \Leftrightarrow\ &
    \eta(\sigma)^{-1}\sigma(\X)^\top\eta(\sigma)\PP^{-\top}\J_{2n}\PP^{-1}\X=\PP^{-\top}\J_{2n}\PP^{-1},\\
    \textrm{since}\ &\textrm{$s(\sqrt{d})=\sqrt{d}$}\\
    \Leftrightarrow\
    &\sigma(\eta(\sigma)\X\eta(\sigma)^{-1})^\top\PP^{-\top}\J_{2n}\PP^{-1}\X=\PP^{-\top}\J_{2n}\PP^{-1}\\
    \Leftrightarrow\
    &\sigma^*(\X)^\top\PP^{-\top}\J_{2n}\PP^{-1}\X=\PP^{-\top}\J_{2n}\PP^{-1},\\
    \textrm{since}\ &\sigma(\eta(\sigma)\X\eta(\sigma)^{-1})=\sigma^*(\X)\\
    \Leftrightarrow\
    &\sigma^*(\K\X\K^{-1})^\top\K\PP^{-\top}\J_{2n}\PP^{-1}\X=\K\PP^{-\top}\J_{2n}\PP^{-1}\\
    \Leftrightarrow\ &\sigma^*(\overline{\X})^t\J_{2n}^*\X=\J_{2n}^*,\\
    \textrm{where}\ &\J_{2n}^*=\K\PP^{-\top}\J_{2n}\PP^{-1}.
\end{align*}
The matrix $\J_{2n}^*$ is a $\overline{\phantom{s}}\otimes\sigma$-Hermitian matrix. Let $\N$ be the matrix introduced in the proof of Proposition \ref{proplatticeSp}. Since $\sigma^*(\N)=\N$, $\sigma^*(\overline{\N}^t)\J_{2n}^*\N$ is the diagonal matrix defined by
\begin{equation*}
    (\sigma^*(\overline{\N})^t\J_{2n}^*\N)_{ii}=\left\{\begin{array}{ll}
        -4(2n-i-1)!i! & \mbox{if $i$ is odd}\\
        -4(2n-i)!(i-1)! & \mbox{if $i$ is even.}
    \end{array}\right.
\end{equation*} We can see in \S7 of \cite{Lewis_IsometryClassificationHermitianForms} that $\tc\otimes\sigma$-Hermitian forms over $A\otimes_F F(\sqrt{d})$ are classified by their rank, their signatures at real places of $F$ and their discriminant. Hence $\J^*_{2n}$ is equivalent to $-\I_{2n}$. We conclude that
$$\prescript{}{\zeta}{\SL_{2n}}(F)\simeq\SU(\I_n,\tc\otimes\sigma;A\otimes_F F(\sqrt{d})).$$ Proposition \ref{propmorphismofalgebraicgroups} shows that $\tau_{2n}(\Gamma)$ is contained in $$\prescript{}{\zeta}{\SL_{2n}}(\mathcal{O}_F)\simeq\SU(\I_n,\overline{\phantom{s}}\otimes\sigma;\mathcal{O}\otimes\mathcal{O}_F[\sqrt{d}])$$ up to finite index. This is a lattice of $\SL(2n,\mathds{R})$ if and only if $d$ is positive exactly at the real place where $A$ splits.

The converse statement is proven as in the proof of Proposition \ref{proplatticeSO}.
\end{proof}

\subsection{Proof of Theorem \ref{theoclassificationfuchsian}}

\begin{proof}[Proof of Theorem \ref{theoclassificationfuchsian}]
Let $G$ be $\SO(k+1,k)$, $\Sp(2n,\mathds{R})$, $\G_2$ or $\SL(n,\mathds{R})$ with $k\geq2$ and $n\geq3$. Let $\Lambda$ be a uniform lattice of $G$ that contains the image of a Fuchsian representation $\rho:\pi_1(S_g)\to G$ for some $g\geq2$. By definition, $\rho=\tau_n\circ j$ for a suitable $n$ where $j$ is a discrete and faithful embedding of $\pi_1(S_g)$ into $\SL(2,\mathds{R})$. Hence $j(\pi_1(S_g))$ is a lattice in $\SL(2,\mathds{R})$.

By Margulis' Arithmeticity Theorem (Theorem 16.3.1 in \cite{Morris_IntroductionArithmeticGroups}), $\Lambda$ is arithmetic. By Proposition \ref{propGsimplearithmetic}, there exists a totally real number field $F$ such that $\Lambda$ is commensurable with the $\mathcal{O}_F$-points of an $F$-algebraic group which is compact at all real places of $F$ except one. Applying Proposition \ref{propmorphismarithmeticlattices} to $\tau_n(\GL(2,\mathds{R}))<G$, the same holds for $j(\pi_1(S_g))$.

Hence the Theorem follows from Propositions \ref{proplatticeSO}, \ref{proplatticeG2}, \ref{proplatticeSLodd}, \ref{proplatticeSp} and \ref{proplatticeSLeven}.

\end{proof}

\section{Construction of Zariski-dense surface subgroups}

\subsection{Reminders on bending}
\label{subsectionremindersbending}

Let $n\geq3$. Let $\rho:\pi_1(S_g)\to\SL(n,\mathds{R})$ be a Fuchsian representation. Pick $\gamma\in\pi_1(S_g)$ which is represented by a simple closed curve which separates $S_g$ into two surfaces $C$ and $D$. Then $\pi_1(S_g)=\pi_1(C)*_{\gamma}\pi_1(D)$. Let $\B\in\SL(n,\mathds{R})$ such that $\B$ commutes with $\rho(\gamma)$. Note that $\rho(\gamma)$ is diagonalizable over $\mathds{R}$ with distinct eigenvalue. The restriction of $\rho$ to $\pi_1(C)$ together with the conjugate by $\B$ of the restriction of $\rho$ to $\pi_1(D)$ induce a new representation
$$\rho_{\B}:\pi_1(S_g)\to\SL(n,\mathds{R})$$ called the \emph{bending} of $\rho$ by $\B$.

To construct Zariski-dense Hitchin representations in a lattice $\Lambda$, we bend a Fuchsian representation with image in $\Lambda$ (supposing such a representation exists) by a matrix $\B$ in $\Lambda$. If $\B$ has only positive eigenvalues than there is a continuous path from $\I_n$ to $\B$ within the centralizer $\Comm_{\SL(n,\mathds{R})}(\rho(\gamma))$ of $\rho(\gamma)$, which implies that $\rho_{\B}$ is a Hitchin representation. The next lemma gives us a control on the stabilizer of a curve $\gamma$ in $\Lambda$. Note that the image of a non-trivial element under a Hitchin representation is always regular (see Theorem 1.5 in \cite{Labourie_Anosovflowssurfaceandcurves}). 

\begin{lemma}
\label{centralizercocompactlattice}
Let $G$ be a noncompact semisimple Lie group and let $\Lambda$ be a uniform lattice of $G$. Let $\gamma\in\Lambda$ be regular. Then $\Comm_{\Lambda}(\gamma)$ is an abelian group of rank equal to $\rank_{\mathds{R}}G$.
\end{lemma}

\begin{remark}
The proof of this lemma follows an argument of Hamenstädt that appears in her talk \cite{Hamenstadt_Specialpointscharactervariety}. She attributes the result to Burger and Schroeder. Note that Lemma \ref{centralizercocompactlattice} can be avoided but is stated her for completeness. Indeed one can use Galois cohomology to compute the centralizer of a given simple closed curve in the lattices of interest explicitly. See \S5.2 in \cite{Audibert_Zariskidensesurfacegroups} where it has been done for non-uniform lattices. The same computations work for uniform lattices.
\end{remark}

\begin{proof}
Let $\X$ be the symmetric space of $G$ of nonpositive curvature. Since $\Lambda$ is cocompact, there exists $r>0$ such that for all $p\in\X$, the covering map $\pi:\X\to\X/\Lambda$ is a diffeomorphism between the balls $\B(p,r)$ and $\B(\pi(p),r)$.

Let $\mathfrak{a}$ be the Cartan subalgebra of the Lie algebra of $G$ such that $\gamma\in\exp(\mathfrak{a})$. It is unique since $\gamma$ is regular. Then $\mathfrak{a}$ embeds in $\X$ as a maximal flat. It turns out that $$\{g\in\Lambda\ |\ \Ad(g)\mathfrak{a}=\mathfrak{a}\}=\Comm_{\Lambda}(\gamma)=:\Gamma.$$
For simplicity, we suppose that $\Lambda$ is torsion-free. Consider $\mathfrak{a}/\Gamma$. It is a flat manifold and as such isometric to $\mathds{R}^l\times(\mathds{S}^1)^s$ where $l+s=\dim(\mathfrak{a})=\rank_{\mathds{R}}(G)$. In fact $s$ is the rank of $\Gamma$. Thus we want to show that $l=0$.

Suppose that $l\geq1$. Thus there is a copy of $\mathds{R}$ in $\mathfrak{a}/\Gamma$. Pick a geodesic representative of $\gamma$ in $\mathfrak{a}/\Gamma$. Denote by $\gamma_n$ the translate of $\gamma$ by length $n$ along a fixed copy of $\mathds{R}$. We now see $\mathfrak{a}/\Gamma$ as immersed in $\X/\Lambda$. Since $\X/\Lambda$ is compact Arzelà–Ascoli's Theorem tells us that the sequence of geodesics $\gamma_n$ converges uniformly in $\X/\Lambda$.
Hence there exist $N<M$ such that the Hausdorff distance between $\gamma_N$ and $\gamma_M$ is less that $\frac{r}{2}$.

Pick a lift of $\gamma_N$ in $\mathfrak{a}$ and denote it $\widetilde{\gamma_N}$. Consider a lift of $\gamma_M$ in $\X$ which is at distance less than $\frac{r}{2}$ from $\widetilde{\gamma_N}$ at some point and denote it $\widetilde{\gamma_M}$.
There exists an element $g\in\Lambda$ that sends the $(M-N)$-translate of $\widetilde{\gamma_N}$ to $\widetilde{\gamma_M}$.
The curve $\widetilde{\gamma_M}$ stays in a tubular neighborhood of $\widetilde{\gamma_N}$ of radius $\frac{r}{2}$.
The Flat Strip Theorem (Proposition 5.1. in \cite{Eberlein_Visibilitymanifolds}) implies that $\widetilde{\gamma_N}$ and $\widetilde{\gamma_M}$ lie in a maximal flat of $\X$. Since $\widetilde{\gamma_N}$ is regular, it has to be $\mathfrak{a}$.
Hence $g$ preserves $\mathfrak{a}$ so $g\in\Gamma$. On the other hand, translates of $\gamma$ in $\mathfrak{a}/\Gamma$ are supposed to be at distance at least $1$ from each other. This is a contradiction.
\end{proof}

To control the Zariski-closures of our bent representations, we need the following result. For a proof, see \S4 of \cite{Audibert_Zariskidensesurfacegroups}.

\begin{lemma}
\label{lemmaZariskidensity}
Let $\B\in\SL(n,\mathds{R})$ which commutes with $\rho(\gamma)$ and with positive eigenvalues. Then $\rho_{\B}$ has Zariski-closure one of the following:
\begin{itemize}
    \setlength\itemsep{0cm}
    \item $\tau_n(\SL(2,\mathds{R}))$ if and only if $\B\in\tau_n(\GL(2,\mathds{R}))$,
    \item $\Sp(\J_n,\mathds{R})$ if and only if $n=2k$, $\B\in\Sp(\J_n,\mathds{R})$ and $\B\not\in\tau_n(\GL(2,\mathds{R}))$,
    \item $\SO(\J_{n},\mathds{R})$ if and only if $n=2k+1$, $\B\in\SO(\J_{n},\mathds{R})$, $\B\not\in\tau_n(\GL(2,\mathds{R}))$ and if $n=7$, $\B\not\in\emph{\textbf{G}}_2(\mathds{R})$ ,
    \item \emph{$\textbf{G}_2(\mathds{R})$} if and only if $n=7$, $\B\in\emph{\textbf{G}}_2(\mathds{R})$ and $\B\not\in\tau_7(\GL(2,\mathds{R}))$,
    \item $\SL(n,\mathds{R})$ otherwise.
\end{itemize}
\end{lemma}

Lemma \ref{lemmaZariskidensity} is a consequence of the classification of Zariski-closures of Hitchin representations by Guichard. We recall it here for the sake of completeness.

\begin{theorem}[Guichard \cite{Guichard_Zariskiclosurepositive}, see also Sambarino \cite{Sambarino_InfinitesimalZariskiClosure}]
\label{theoremGuichard}
Let $\rho:\pi_1(S)\to\SL(n,\mathds{R})$ be a Hitchin representation. Then the Zariski-closure of $\rho$ is either $\SL(n,\mathds{R})$, a principal $\SL(2,\mathds{R})$ or conjugated to one of the following:
\begin{itemize}
    \setlength\itemsep{0cm}
    \item $\Sp(2k,\mathds{R})$ if $n=2k$,
    \item $\SO(k+1,k)$ if $n=2k+1$,
    \item \emph{$\textbf{G}_2(\mathds{R})$} if $n=7$.
\end{itemize}
\end{theorem}

\subsection{Strong Approximation Theorem}

In this part we recall the Strong Approximation Theorem and set up some technical lemmas so that we can apply the theorem to the representations we will build. It will be needed in the last part to prove that our constructions provide infinitely many mapping class group orbits of representations.

If $R$ is a ring and $a\in R\setminus\{0\}$ then we denote by $R_a$ the localization of $R$ at the multiplicative set $\{a^n|n\in\mathds{N}\}$. The ideals in $R_a$ are in bijective correspondence with ideals of $R$ that do not contain any power of $a$.

\begin{theorem}[Strong Approximation, Weisfeiler \cite{Weisfeiler_StrongapproximationZariskidensesubgroups}]
\label{theostrongapproximation}
Let $\G$ be a connected $\Qbar$-algebraic group which is almost simple and simply connected. Let $\Gamma$ be a finitely generated Zariski-dense subgroup of $\G(\Qbar)$. Denote by $$R=\mathds{Z}[\Tr(\Ad(\Gamma))].$$

Then there exists $a\in R$, a finite index subgroup $\Gamma'<\Gamma$ and a structure $\G_0$ of a group scheme over $R_a$ on $\G$ such that $\Gamma'<\G_0(R_a)$ and $\Gamma'$ is dense in $$\varprojlim_{|R_a/I|<\infty}\G_0\big(R_a/I\big)$$ where the projective limit is over ideals of $R_a$.
\end{theorem}

\begin{remark}
\label{remarksurjectionstrongapproximation}
In the context of Theorem \ref{theostrongapproximation}, let $I$ be an ideal of $R_a$ such that $R_a/I$ is finite and denote by $$\pi:\G_0(R_a)\to\G_0\big(R_a/I\big)$$ the canonical map. The theorem implies that $\pi(\Gamma')=\G_0(R_a/I)$.
\end{remark}

Theorem \ref{theostrongapproximation} is needed to distinguish mapping class groups orbits of representations. Unfortunately, the ring $R$ depends on $\Gamma$. We thus need to show that, in our setting, this dependence can be removed. This is done in Proposition \ref{propstrongapproximationadapted}.

\begin{definition}
Let $F$ be a number field. An \emph{order} of $F$ is a subring $R\subset\mathcal{O}_F$ which is also a $\mathds{Z}$-module of rank $[F:\mathds{Q}]$. \end{definition}

\begin{lemma}
\label{lemmatraceorder}
Let $F$ be a totally real number field and $\G$ be a connected semisimple $F$-algebraic group which is compact over all real places of $F$ except one. Let $\Gamma<\G(\mathcal{O}_F)$ be a finitely generated subgroup which is Zariski-dense in $\G$. Then $\mathds{Z}[\Tr(\Ad(\Gamma))]$ is an order of $F$.
\end{lemma}

\begin{proof}
Denote by $K=\mathds{Q}(\Tr(\Ad(\Gamma)))\subset F$. We want to show that $K=F$. By Theorem 1 in Vinberg \cite{Vinberg_Ringsdefinitniondensesubgroups}, there is a $\mathds{C}$-basis $\beta$ of $\mathfrak{g}$, the Lie algebra of $\G(\mathds{C})$, such that for all $\gamma\in\Gamma$ $$\Mat_{\beta}(\Ad(\gamma))\in\GL(n,K)\ \textrm{with $n=\dim_{\mathds{C}}\mathfrak{g}$}.$$

Since $\Gamma$ is Zariski-dense in $\G$, $\Ad(\Gamma)$ is Zariski-dense in $\Ad(\G)$. Hence $\Ad(\G(\Qbar))$ is defined over $K$ (see Theorem 14.4 in Chapter AG of Borel \cite{Borel_Linearalgebraicgroups}). Denote by $\sigma$ the real place of $F$ such that $\G(F_{\sigma})$ is non-compact. Suppose $K\neq F$. There is a real place $\iota$ of $F$ different from $\sigma$ such that $\sigma_{|K}=\iota_{|K}$. Since $\G$ is defined over $K$, $\G(F_{\sigma})\simeq\G(F_{\iota})$ while over $\mathds{R}$ one is compact and the other is not. This is a contradiction. Hence $K=F$.

Now $\mathds{Z}[\Tr(\Ad(\Gamma))]$ is a finitely generated torsion free $\mathds{Z}$-module, so it is free. Since
$$\mathds{Z}[\Tr(\Ad(\Gamma))]\otimes_{\mathds{Z}}\mathds{Q}\simeq\mathds{Q}(\Tr(\Ad(\Gamma))),$$ $\mathds{Z}[\Tr(\Ad(\Gamma))]=F$ is of rank $[F:\mathds{Q}]$. Thus it is an order of $F$.
\end{proof}

\begin{lemma}
\label{lemmaprimeidealsinorder}
Let $R\subset\mathcal{O}_F$ be an order of a number field $F$. Then for all but finitely many prime ideals $\mathfrak{p}\subset\mathcal{O}_F$ 
\begin{equation}
\label{equationprimeideal}
    R\!\raisebox{-.65ex}{\ensuremath{/R\cap\mathfrak{p}}}=\mathcal{O}_F\!\raisebox{-.65ex}{\ensuremath{/\mathfrak{p}}}.
\end{equation}
\end{lemma}

\begin{proof}
Suppose not. Since $R$ is an order of $F$, the index of $R$ in $\mathcal{O}_F$ is finite. Denote it by $n$. Let $\mathfrak{p}\subset\mathcal{O}_F$ be a prime ideal that does not satisfy \eqref{equationprimeideal} such that its residue field has characteristic $p>n$. It exists since for every prime number, there is only finitely many prime ideals in $\mathcal{O}_F$ that contain it.

One has
\begin{equation*}
    R\!\raisebox{-.65ex}{\ensuremath{/R\cap\mathfrak{p}}}\subset\mathcal{O}_F\!\raisebox{-.65ex}{\ensuremath{/\mathfrak{p}}}.
\end{equation*} As a subring of a finite field, $R/R\cap\mathfrak{p}$ is a field. Recall that
\begin{equation*}
    R\mathfrak{p}\!\raisebox{-.65ex}{\ensuremath{/\mathfrak{p}}}\simeq R\!\raisebox{-.65ex}{\ensuremath{/R\cap\mathfrak{p}}}.
\end{equation*}
Hence
\begin{equation*}
    \left(\mathcal{O}_F\!\raisebox{-.65ex}{\ensuremath{/\mathfrak{p}}}\right)\big/\left(R\!\raisebox{-.65ex}{\ensuremath{/R\cap\mathfrak{p}}}\right)\simeq\left(\mathcal{O}_F\!\raisebox{-.65ex}{\ensuremath{/\mathfrak{p}}}\right)\big/\left(R\mathfrak{p}\!\raisebox{-.65ex}{\ensuremath{/\mathfrak{p}}}\right)\simeq\mathcal{O}_F\!\raisebox{-.65ex}{\ensuremath{/R\mathfrak{p}}}.
\end{equation*}
The latter must contain at least $p$ elements, so $\mathcal{O}_F/R$ contains at least $p$ elements. This is a contradiction.
\end{proof}

The following proposition is an adaptation of the Strong Approximation Theorem, Theorem \ref{theostrongapproximation}, to our context.

\begin{proposition}
\label{propstrongapproximationadapted}
Let $F$ be a totally real number field and $\G$ be a connected almost simple and simply connected $F$-algebraic group which is compact over all real places of $F$ except one. Let $\Gamma<\G(F)$ be a finitely generated Zariski-dense subgroup. There exists a finite index subgroup $\Gamma'<\Gamma$, $a\in\mathcal{O}_F$ and a group scheme structure $\G_0$ over $({\mathcal{O}_F})_a$ on $\G$ such that $\Gamma'<\G_0(({\mathcal{O}_F})_a)$ and for all prime ideals $\mathfrak{p}$ of $\mathcal{O}_F$ but finitely many, $\Gamma'$ surjects onto $$\G_0\left(\mathcal{O}_F\!\raisebox{-.65ex}{\ensuremath{/\mathfrak{p}}}\right).$$
\end{proposition}

\begin{remark}
The above makes sense since for all prime ideals $\mathfrak{p}$ of $\mathcal{O}_F$ that do not contain $a$, $$({\mathcal{O}_F})_a\!\raisebox{-.65ex}{\ensuremath{/\mathfrak{p}_a}}\simeq\mathcal{O}_F\!\raisebox{-.65ex}{\ensuremath{/\mathfrak{p}}}.$$
\end{remark}

\begin{proof}
Let $R=\mathds{Z}[\Tr(\Ad(\Gamma))]$. By the Strong Approximation Theorem, Theorem \ref{theostrongapproximation}, there exists $a\in R$, a finite index subgroup $\Gamma'<\Gamma$ and a group scheme structure $\G_0$ over $R_a$ on $\G$ such that $\Gamma'<\G_0(R_a)$ and $\Gamma'$ is dense in
$$\varprojlim_{|R_a/I|<\infty}\G_0\big(R_a\!\raisebox{-.65ex}{\ensuremath{/I}}\big).$$ Since $R\subset\mathcal{O}_F$ we can also consider $\G_0$ as a group scheme over $({\mathcal{O}_F})_a$. Let $\mathfrak{p}$ be a prime ideal in $\mathcal{O}_F$ that does not contain $a$ and such that
\begin{equation*}
    R\!\raisebox{-.65ex}{\ensuremath{/R\cap\mathfrak{p}}}=\mathcal{O}_F\!\raisebox{-.65ex}{\ensuremath{/\mathfrak{p}}}.
\end{equation*}
Since $R$ is an order of $F$ (see Lemma \ref{lemmatraceorder}), Lemma \ref{lemmaprimeidealsinorder} implies that there is only finitely many prime ideals that do not satisfy both of those assumptions. As explained in Remark \ref{remarksurjectionstrongapproximation}, $\Gamma'$ surjects onto
\begin{equation*}
    \G_0\big(R_a\!\raisebox{-.65ex}{\ensuremath{/(R\cap\mathfrak{p})_a}}\big)\simeq\G_0\big(R\!\raisebox{-.65ex}{\ensuremath{/(R\cap\mathfrak{p})}}\big)\simeq\G_0\big(\mathcal{O}_F\!\raisebox{-.65ex}{\ensuremath{/\mathfrak{p}}}\big).
\end{equation*}
\end{proof}

\subsection{Proof of Theorem \ref{theoprincipalthinHitchin}}

Recall that we would like to prove that every lattice $\Lambda$ in $G$ as in Table \ref{Table1} contains infinitely many mapping class group orbits of Zariski-dense Hitchin representations, for some fixed genus.

\begin{proof}[Proof of Theorem \ref{theoprincipalthinHitchin} for $G=\Sp(2n,\mathds{R})$]
Let $n\geq2$. Let $\Lambda<\Sp(2n,\mathds{R})$ be a uniform lattice. By Proposition \ref{propclassificationlatticeSp} there exists a quaternion algebra $A$ over a totally real number field $F\neq\mathds{Q}$ and an order $\mathcal{O}$ of $A$ such that $\Lambda$ is widely commensurable with $\SU(\I_n,\tc;\mathcal{O})$. Proposition \ref{proplatticeSp} shows that $\tau_{2n}(\mathcal{O}^1)$ lies in a lattice widely commensurable with $\Lambda$.

Let $S$ be the quotient $\mathds{H}^2/\mathcal{O}^1$. Up to finite cover, it is a closed surface of genus at least $2$. Denote by $\rho$ the Fuchsian representation of $\pi_1(S)$ induced by $\tau_{2n}$. Let $\gamma$ be a simple closed separating curve on $S$. By Lemma \ref{centralizercocompactlattice}, there exists $\B\in\Lambda$ which commutes with $\rho(\gamma)$, has only positive eigenvalues and such that $\B\not\in\tau_{2n}(\GL(2,\mathds{R}))$. The bent representation $\rho_{\B}$ of $\rho$ by $\B$ along $\gamma$ has Zariski-dense image in $\Sp(2n,\mathds{R})$ (see Lemma \ref{lemmaZariskidensity}).

Furthermore the sequence of representations $\rho_{\B^k}$ are Zariski-dense for all $k\geq1$ and give rise to infinitely many $\MCG(S)$-orbits of Hitchin representations, as we will show now. Let $k\geq1$ and denote by $\Gamma_k$ the image of $\rho_{\B^k}$. Since $F$-forms of $\Sp_{2n}$ are simply connected, the Strong Approximation Theorem (Theorem \ref{theostrongapproximation} and Proposition \ref{propstrongapproximationadapted}) implies that for all but finitely many prime ideals $\mathfrak{p}$ of $\mathcal{O}_F$, a finite index subgroup $\Gamma_k'<\Gamma_k$ surjects onto $\Sp(2n,\mathcal{O}_F/\mathfrak{p})$. Lemma 5.10 in \cite{Audibert_Zariskidensesurfacegroups} shows that $$\Tr(\Gamma_k')\equiv\mathcal{O}_F\mod\mathfrak{p}$$ for all but finitely many $\mathfrak{p}$.

Suppose that $\{\rho_{\B^k},k\geq1\}$ lie in a finite number of orbits under $\MCG(S)$. Note that there exists $\PP\in\mathds{Z}[\X]$ such that $\Tr(\tau_{2n}(\M))=\PP(\Tr(\M))$ for all $\M$. By our assumption, there is a prime ideal $\mathfrak{p}$ for which the reduction of $\Tr(\Gamma_k')$ is surjective for all $k\geq1$ and the map
$$\mathcal{O}_F\!\raisebox{-.65ex}{\ensuremath{/\mathfrak{p}}}\to\mathcal{O}_F\!\raisebox{-.65ex}{\ensuremath{/\mathfrak{p}}},\ x\mapsto\PP(x)$$ is not a surjection (see Corollary 1.8 in \cite{Shallue_PermutationPolynomialsFiniteFields}). Let $k\geq1$ such that $\B^k$ is trivial modulo $\mathfrak{p}$. Then $$\Tr(\Gamma_k')=\Tr(\rho(\pi_1(S)))=\PP(\Tr(\mathcal{O}^1))\mod\mathfrak{p}$$ which is not $\mathcal{O}_F/\mathfrak{p}$. This is a contradiction.
\end{proof}

\begin{proof}[Proof of Theorem \ref{theoprincipalthinHitchin} for $G=\SL(2k+1,\mathds{R})$]
Let $n=2k+1\geq3$ be odd. Let $F\neq\mathds{Q}$ be a totally real number field and $d\in\mathcal{O}_F$ which is positive at exactly one real place of $F$. Define $A=(d,d)_F$. Let $\mathcal{O}$ be an order in $A$. Proposition \ref{proplatticeSLodd} implies that $\tau_n(\mathcal{O}^1)$ lies in a subgroup $\Lambda$ of $\SL(n,\mathds{R})$ which widely commensurable with $\SU(\I_n,\sigma;\mathcal{O}_F[\sqrt{d}])$ where $\sigma\in\Gal(F(\sqrt{d})/F)$ is non-trivial.

The quotient $S:=\mathds{H}^2/\mathcal{O}^1$ is a closed surface of genus at least $2$ up to finite cover. Let $\rho$ be the Fuchsian representation of $\pi_1(S)$ induced by $\tau_{2n}$. Let $\gamma$ be a simple closed separating curve on $S$. By Lemma \ref{centralizercocompactlattice}, there exists $\B\in\Lambda$ which commutes with $\rho(\gamma)$, has only positive eigenvalues and such that $\B\not\in\SO(\J_n,\mathds{R})$. The bent representation $\rho_{\B}$ of $\rho$ by $\B$ along $\gamma$ has Zariski-dense image in $\SL(n,\mathds{R})$ (see Lemma \ref{lemmaZariskidensity}).

We now show that the sequence of representations $\rho_{\B^l}$ give rise to infinitely many $\MCG(S)$-orbits. Let $l\geq1$ and denote by $\Gamma_l$ the image of $\rho_{\B^l}$. Since $F$-forms of $\SL_n$ are simply connected, the Strong Approximation Theorem (Theorem \ref{theostrongapproximation} and Proposition \ref{propstrongapproximationadapted}) implies that for all but finitely many prime ideals $\mathfrak{p}$ of $\mathcal{O}_F$, a finite index subgroup $\Gamma_l'<\Gamma_l$ surjects onto
\begin{equation*}
    \left\{\begin{array}{ll}
        \SL_n\big(\mathcal{O}_F\!\raisebox{-.65ex}{\ensuremath{/\mathfrak{p}}}\big) & \mbox{if $d$ is a square in $\mathcal{O}_F\!\raisebox{-.65ex}{\ensuremath{/\mathfrak{p}}}$} \\
        \SU\big(\I_n,\sigma_0;\mathcal{O}_F\!\raisebox{-.65ex}{\ensuremath{/\mathfrak{p}}}[\sqrt{d}]\big) & \mbox{if $d$ is not a square in $\mathcal{O}_F\!\raisebox{-.65ex}{\ensuremath{/\mathfrak{p}}}$}
    \end{array}\right.
\end{equation*} for $$\sigma_0\in\Gal\big(\mathcal{O}_F\!\raisebox{-.65ex}{\ensuremath{/\mathfrak{p}}}(\sqrt{d})/\mathcal{O}_F\!\raisebox{-.65ex}{\ensuremath{/\mathfrak{p}}}\big)$$ non-trivial. See \S 5.3 of \cite{Audibert_Zariskidensesurfacegroups} for details on this dichotomy. Lemmas 5.6 and 5.7 in \cite{Audibert_Zariskidensesurfacegroups} show that
\begin{equation*}
    \Tr(\Gamma_l')=\left\{\begin{array}{ll}
    \mathcal{O}_F\!\raisebox{-.65ex}{\ensuremath{/\mathfrak{p}}}\mod\mathfrak{p} & \mbox{if $d$ is a square in $\mathcal{O}_F\!\raisebox{-.65ex}{\ensuremath{/\mathfrak{p}}}$} \\
    \mathcal{O}_F\!\raisebox{-.65ex}{\ensuremath{/\mathfrak{p}}}[\sqrt{d}]\mod\mathfrak{p} & \mbox{if $d$ is not a square in $\mathcal{O}_F\!\raisebox{-.65ex}{\ensuremath{/\mathfrak{p}}}$}
\end{array}\right.
\end{equation*} for all but finitely many prime ideals $\mathfrak{p}$.

Suppose that $\{\rho_{\B^l},l\geq1\}$ lie in a finite number of orbits under $\MCG(S)$. As before, there exists $\PP\in\mathds{Z}[\X]$ such that $\Tr(\tau_n(\M))=\PP(\Tr(\M))$ for all $\M$. By our assumption, there is a prime ideal $\mathfrak{p}$ such that the reduction of $\Tr(\Gamma'_l)$ is surjective for all $l\geq1$ and the map
$$\mathcal{O}_F\!\raisebox{-.65ex}{\ensuremath{/\mathfrak{p}}}\to\mathcal{O}_F\!\raisebox{-.65ex}{\ensuremath{/\mathfrak{p}}},\ x\mapsto\PP(x)$$ is not a surjection (see Corollary 1.8 in \cite{Shallue_PermutationPolynomialsFiniteFields}). Let $l\geq1$ such that $\B^l$ is trivial modulo $\mathfrak{p}$. Then
$$\Tr(\Gamma_l')=\Tr(\rho(\pi_1(S)))=\PP(\Tr(\mathcal{O}^1))\mod\mathfrak{p}.$$This is a contradiction.
\end{proof}

\begin{proof}[Proof of Theorem \ref{theoprincipalthinHitchin} for $G=\SL(2n,\mathds{R})$]
Let $n\geq2$. Let $F\neq\mathds{Q}$ be a totally real number field. Let $A$ be a quaternion algebra over $F$ which splits at exactly one real place $\sigma$ among the real places of $F$ and $\mathcal{O}$ be an order of $A$. Let $d\in\mathcal{O}_F$ which is positive exactly at $\sigma$. Proposition \ref{proplatticeSLeven} shows that $\tau_{2n}(\mathcal{O}^1)$ lies in a lattice $\Lambda$ of $\SL(2n,\mathds{R})$ widely commensurable with $\SU(\I_n,\tc\otimes\sigma;\mathcal{O}\otimes\mathcal{O}_F[\sqrt{d}])$. The proof now follows exactly the argument of the proof of Theorem \ref{theoprincipalthinHitchin} for $G=\SL(2k+1,\mathds{R})$.
\end{proof}

\begin{proof}[Proof of Theorem \ref{theoprincipalthinHitchin} for $G=\emph{\textbf{G}}_2(\mathds{R})$]
Let $\Lambda$ be a uniform lattice in $\textbf{G}_2(\mathds{R})$. By Proposition \ref{propclassificationlatticeG2}, there exists a totally real number field $F\neq\mathds{Q}$ and $a,b\in \mathcal{O}_F$ such that for all embeddings $\sigma:F\to\mathds{R}$ except one, $\sigma(a)$ and $\sigma(b)$ are negative and $\Lambda$ is widely commensurable with $\textbf{G}_2^{a,b}(\mathcal{O}_F)$. Let $A=(a,b)_F$ and $\mathcal{O}$ be an order of $A$. Proposition \ref{proplatticeG2} shows that $\tau_7(\mathcal{O}^1)$ lies in $\Lambda$ up to wide commensurability.

The proof follows the argument of Theorem \ref{theoprincipalthinHitchin} for $G=\Sp(2n,\mathds{R})$, using the following lemma.
\end{proof}

\begin{lemma}
Let $p\neq2$ be a prime and let $q=p^n$.  Let $\G_2(\mathds{F}_q)$ be the finite group of Lie type associated to the Dynkin diagram $\G_2$ embedded in $\SL_7(\mathds{F}_q)$ via its $7$-dimensional irreducible representation. Then $\Tr(\G_2(\mathds{F}_q))=\mathds{F}_q$.
\end{lemma}

\begin{proof}
The group $\G_2(\mathds{F}_q)$ is the automorphism group of the unique octonion algebra $\mathds{O}$ over $\mathds{F}_q$, see \S1.10 of \cite{Springer_OctonionJordanAlgebrasExceptionalGroups}. We can describe it as follows. As an $\mathds{F}_q$-vector space $\mathds{O}=\M_2(\mathds{F}_q)\oplus\M_2(\mathds{F}_q)$. Let $\overline{\phantom{s}}$ denote the conjugation of the quaternion algebra $\M_2(\mathds{F}_q)$. For any $\A\in\M_2(\mathds{F}_q)$, $\overline{\A}$ is the transpose of the cofactor matrix of $\A$. The multiplication is defined by $$(\A_1,\B_1)\cdot(\A_2,\B_2)=(\A_1\A_2-\overline{\B_2}\B_1,\B_2\A_1+\B_1\overline{\A_2}),$$ see \S1.5 in \cite{Springer_OctonionJordanAlgebrasExceptionalGroups}. We endow $\mathds{O}$ with the quadratic form $$(\A,\B)\mapsto\Det(\A)+\Det(\B).$$ An automorphism of $\mathds{O}$ induces an element of $\G_2(\mathds{F}_q)$ by its restriction to the orthogonal of $\langle(\I_2,0)\rangle$.

Let $a\in\mathds{F}_q$ and consider the following automorphism
$$\varphi_{a}:\mathds{O}\to\mathds{O},\ (\A,\B)\to(\A,\X\B)\ \textrm{where} \X=\begin{pmatrix}
a&1\\
-1&0
\end{pmatrix}.$$ It is an automorphism of $\mathds{O}$ and thus defines an element of $\G_2(\mathds{F}_q)$. Picking a basis of $\mathds{O}$, computations show that it has trace $a+3$. Hence any element of $\mathds{F}_q$ is the trace of an element of $\G_2(\mathds{F}_q)$.
\end{proof}

Since $F$-forms of $\SO_n$ are not simply connected, we cannot apply the Strong Approximation Theorem as stated in Proposition \ref{propstrongapproximationadapted}. We will instead use Theorem 5.1 in Nori \cite{Nori_SubgroupsofGLn(Fp)}. Before doing so, we need to prove the following lemma.

\begin{lemma}
\label{lemmaZariskidensityinproducts}
    Let $\Gamma$ be a group and $G_1,\dots,G_n$ be centerless connected simple Lie groups. For all $1\leq i\leq n$, let $\rho_i:\Gamma\to G_i$ be a Zariski-dense representation. Suppose that there does not exist $i\neq j$ and a continuous isomorphism $\phi: G_i\to G_j$ satisfying $\phi\circ\rho_i=\rho_j$. Then $$\rho_1\times\dots\times\rho_n:\Gamma\to G_1\times\dots\times G_n$$ is Zariski-dense.
\end{lemma}

\begin{proof}
We will show the lemma by induction on $n\geq1$. The case $n=1$ is clear. Let $n\geq2$ and assume the result has been shown for $n-1$. Denote by $H$ the Zariski-closure of the image of $\rho\coloneqq\rho_1\times\dots\times\rho_n$. We want to show that $H$ is equal to $G\coloneqq G_1\times\dots\times G_n$.

Let $\pi_i:G\to G_i$ be the projection onto the $i$-th coordinate. By Lemma 3.1 in \cite{DalBo_Markedlengthrigidityforsymmetricspaces}, $\pi_i(H)=G_i.$ Let $p_i:G\to G/G_i$ be the projection that forgets about the $i$-th coordinate. Using the induction hypothesis and Lemma 3.1 in \cite{DalBo_Markedlengthrigidityforsymmetricspaces}, $p_i(H)=G/G_i$.
    
Denote by $$H_i=H\cap(\{e\}\times...\times\{e\}\times G_i\times\{e\}\times...\times\{e\})\xhookrightarrow{\pi_i} G_i.$$ Let $h\in H_i$ and $g\in G_i$. There exists $k\in H$ such that $\pi_i(k)=g$. Hence $$g\pi_i(h)g^{-1}=\pi_i(khk^{-1})\in\pi_i(H_i)$$ which proves that $\pi_i(H_i)$ is normal in $G_i$.
    
Suppose that for some $1\leq i\leq n$, $\pi_i(H_i)=G_i$. Denote by $$K_i=H\cap(G_1\times\dots\times G_{i-1}\times\{e\}\times G_{i+1}\times\dots\times G_n)\xhookrightarrow{p_i}G/G_i.$$ Pick $g\in G/G_i$. There exists $h\in H$ such that $p_i(h)=g$. Denote by $h_i$ its $i$-th coordinate. There exists $k\in H_i$ such that $\pi_i(k)=h_i$. Hence $hk^{-1}\in K_i$ and is send to $g$ under $p_i$. This shows that $K_i\hookrightarrow G/G_i$ is an isomorphism. Finally $H=G$.

Otherwise, suppose that $\pi_i(H_i)$ is trivial for all $i$. Pick $g\in G/G_n$. There exists a unique $\phi(g)\in G_n$ such that $(g,\phi(g))\in H$. Indeed if $(g,h)$ and $(g,h')$ are in $H$ then $(e,h^{-1}h')\in \pi_n(H_n)$ which is trivial. If follows from $$(g,\phi(g))(g',\phi(g'))=(gg',\phi(g)\phi(g'))$$ that $\phi$ is a group homomorphism. Since $\pi_n(H)=G_n$, $\phi$ is surjective.

For each $1\leq i\leq n-1$, the restriction of $\phi$ to $G_i$ is either trivial or an isomorphism. Indeed the image of $G_i$ under $\phi$ is a normal subgroup of $G_n$. Hence there exists $1\leq i\leq n-1$ such that $\phi_{|G_i}:G_i\to G_n$ is an isomorphism. Since its graph is closed, it is continuous. The same holds also for its inverse. Finally $\phi_{|G_i}\circ\rho_i=\rho_j$. Contradiction.
\end{proof}

\begin{proof}[Proof of Theorem \ref{theoprincipalthinHitchin} for $G=\SO(k+1,k)$]
Let $n=2k+1\geq5$ with $k\equiv1,2[4]$. Let $\Lambda$ be a uniform lattice in $\SO(k+1,k)$. Proposition \ref{propclassificationlatticeSO} shows that it is widely commensurable to $\SO(Q,\mathcal{O}_F)$ for $F\neq\mathds{Q}$ a totally real number field, $Q\in\SL(n,F)$ a symmetric matrix which is positive definite at all real places of $F$ except one such real place $\sigma$ for which $Q$ has signature equal to $(k+1,k)$ if $k$ is even and $(k,k+1)$ if $k$ is odd.

We want to apply Proposition \ref{proplatticeSO} to an order of a suitable quaternion algebra over $F$. Denote by $S$ the set of finite places $\mathcal{P}$ of $F$ such that $$\mathcal{E}_{\mathcal{P}}(Q)\otimes(-1,-1)_{\mathcal{P}}\neq1.$$ We show that $S\cup V_F\setminus\{\sigma\}$ has even cardinality. By Hilbert's Reciprocity Law (see Theorem 0.9.10 in \cite{Maclachlan_ArithmeticHyperbolic3Manifolds}), $\mathcal{E}_{v}(Q)\otimes(-1,-1)_{v}\neq1$ at an even number of places $v$ of $F$. The number of infinite places $v\in V_F$ where $\mathcal{E}_{v}(Q)\otimes(-1,-1)_{v}\neq1$ is $|V_F|-1$ since \begin{equation*}
    \mathcal{E}_{v}(Q)=\left\{\begin{array}{ll}
        1 & \mbox{if $v\neq\sigma$} \\
        -1 & \mbox{if $v=\sigma$.}
    \end{array}\right.
\end{equation*}
Hence $|S|\equiv|V_F|-1\ [2]$ which implies that $S\cup V_F\setminus\{\sigma\}$ has even cardinality.

Theorem 7.3.6 in \cite{Maclachlan_ArithmeticHyperbolic3Manifolds} implies that there exists a quaternion algebra $A$ over $F$ that does not split exactly at $S\cup V_F\setminus\{\sigma\}$. Let $\mathcal{O}$ be an order of $A$. Proposition \ref{proplatticeSO} tells us that $\tau_n(\mathcal{O}^1)$ lies in a subgroup of $\SO(\J_n,\mathds{R})$ which, up to wide commensurability, we can assume to be $\SO(Q,\mathcal{O}_F)$.

Let $S=\mathds{H}^2/\mathcal{O}^1$. Up to finite cover, we can assume that $S$ is a closed surface of genus at least 2. Let $\gamma$ be a simple closed separating curve on $S$. The map $\tau_n$ induces a Fuchsian representation $\rho:\pi_1(S)\to\SO(\J_n,\mathds{R})$.
Lemma \ref{centralizercocompactlattice} implies that there exists $\B\in\SO(Q,\mathcal{O}_F)$ which commutes with $\rho(\gamma)$, has only positive eigenvalues but which is not in $\tau_n(\GL(2,\mathds{R}))$. The bent representation $\rho_{\B}$ of $\rho$ by $\B$ along the curve $\gamma$ has Zariski-dense image in $\SO(\J_n,\mathds{R})$ as Lemma \ref{lemmaZariskidensity} shows.
Furthermore the sequence of representations $\rho_{\B^l}$ are Zariski-dense for all $l\geq1$ and give rise to infinitely many $\MCG(S)$-orbits, as we will show now. Fix $l\geq1$ and let $\Gamma_l=\rho_{\B^l}(\pi_1(S))$.

For every prime ideal $\mathfrak{p}$ of $\mathcal{O}_F$, denote by
\begin{equation*}
    \pi_{\mathfrak{p}}:\SO(Q,\mathcal{O}_F)\to\SO(Q,\mathcal{O}_F/\mathfrak{p})
\end{equation*}
the reduction modulo $\mathfrak{p}$.
Let $\Omega(Q,\mathcal{O}_F/\mathfrak{p})$ be the commutator subgroup of $\SO(Q,\mathcal{O}_F/\mathfrak{p})$. We first prove that $\pi_{\mathfrak{p}}(\Gamma_l)$ contains $\Omega(Q,\mathcal{O}_F/\mathfrak{p})$ for every prime ideal $\mathfrak{p}$ except a finite number of them. Up to multiplying $Q$ be a constant, we can assume that it has coefficients in $\mathcal{O}_F$. Hence $\SO(Q)$ is an $\mathcal{O}_F$-group scheme. Denote it by $\G$ to simplify the exposition. Now $\Gamma_l<\G(\mathcal{O}_F)\simeq\Res_{\mathcal{O}_F/\mathds{Z}}\G(\mathds{Z})$ which is a subgroup of $\GL_{nd}(\mathds{Z})$ where $d$ is the degree of $F$. We have
\begin{align*}
    \Gamma_l&\hookrightarrow\Res_{\mathcal{O}_F/\mathds{Z}}\G(\mathds{R})\simeq\SO(\sigma_1(Q),\mathds{R})\times\dots\times\SO(\sigma_d(Q),\mathds{R})\\
    \gamma&\mapsto(\sigma_1(\gamma),\dots,\sigma_d(\gamma))
\end{align*}
where the $\sigma_i$ are the embeddings of $F$ in $\mathds{R}$, see \S 2.1.2 in \cite{Platonov_AlgebraicgroupsNumbertheory}. Since $\Gamma_l$ is Zariski-dense in $\SO(Q,\mathcal{O}_F)$, it is Zariski-dense in $\SO(\sigma_i(Q),\mathds{R})$ for every embedding $\sigma_i:F\hookrightarrow\mathds{R}$. Suppose that there exists $i\neq j$ such that $\sigma_i\circ\rho_{\B^l}$ is conjugated to $\sigma_j\circ\rho_{\B^l}$. Then for all $x\in \Tr(\Gamma_l)$, $\sigma_i(x)=\sigma_j(x)$. We showed in Lemma \ref{lemmatraceorder} that $F=\mathds{Q}(\Tr(\Ad(\Gamma_l)))\subset\mathds{Q}(\Tr(\Gamma_l))$. Hence $\sigma_i$ and $\sigma_j$ agree on $F$, which is not possible. By Lemma \ref{lemmaZariskidensityinproducts}, $\Gamma_l$ is Zariski-dense in $\Res_{\mathcal{O}_F/\mathds{Z}}\G$.\footnote{Note that $\Gamma_l$ lies in the connected component of $\SO(k+1,k)$.}

Let $m\geq1$ and $p$ be a prime. For any subgroup $A<\GL_m(\mathds{F}_p)$, we denote by $A^{+}$ the subgroup of $A$ generated by the elements of $A$ of order $p$. It is a normal subgroup. Let $$r_p:\GL_m(\mathds{Z})\to\GL_m(\mathds{F}_p)$$ be the reduction map. Theorem 5.1 in \cite{Nori_SubgroupsofGLn(Fp)} shows that for all but finitely many primes $p$
$$\Res_{\mathcal{O}_F/\mathds{Z}}\G(\mathds{F}_p)^{+}<r_p(\Gamma_l)<\Res_{\mathcal{O}_F/\mathds{Z}}\G(\mathds{F}_p).$$
Recall that for any prime $p$ $$\Res_{\mathcal{O}_F/\mathds{Z}}\G(\mathds{F}_p)\simeq\G(\mathds{F}_p\otimes_{\mathds{Z}}\mathcal{O}_F)\simeq\G(\mathcal{O}_F/p\mathcal{O}_F)\simeq\prod_{\mathfrak{p}|p}\G(\mathcal{O}_F/\mathfrak{p})$$ where $\mathfrak{p}$ are prime ideals in $\mathcal{O}_F$. For almost all prime ideals $\mathfrak{p}$ $$\G(\mathcal{O}_F/\mathfrak{p})\simeq\SO(\I_n,\mathcal{O}_F/\mathfrak{p})$$ since for almost all prime ideals $\mathfrak{p}$, $Q$ is non-degenerate over $\mathcal{O}_F/\mathfrak{p}$ and is thus equivalent to $\I_n$ or to $\lambda\I_n$ where $\lambda$ is not a square in $\mathcal{O}_F/\mathfrak{p}$. Now
$$\Res_{\mathcal{O}_F/\mathds{Z}}\G(\mathds{F}_p)^{+}\simeq\prod_{\mathfrak{p}|p}\SO(\I_n,\mathcal{O}_F/\mathfrak{p})^{+}=\prod_{\mathfrak{p}|p}\Omega(\I_n,\mathcal{O}_F/\mathfrak{p})$$ since $\Omega(\I_n,\mathcal{O}_F/\mathfrak{p})$ is a normal and simple index 2 subgroup of $\SO(\I_n,\mathcal{O}_F/\mathfrak{p})$, see Chapter 3 \S5 in \cite{Suzuki_GroupTheoryI}. Thus
$$r_p(\Gamma_l)\simeq\prod_{\mathfrak{p}|p}\pi_{\mathfrak{p}}(\Gamma_l)\implies\Omega(\I_n,\mathcal{O}_F/\mathfrak{p})<\pi_{\mathfrak{p}}(\Gamma_l)$$ for almost all prime ideals $\mathfrak{p}$. Finally, using Lemma 5.12\footnote{The proof of Lemma 5.12 holds even if the field is not prime.} of \cite{Audibert_Zariskidensesurfacegroups}, we conclude that the reduction of $\Tr(\Gamma_l)$ modulo $\mathfrak{p}$ is surjective for almost all prime ideals $\mathfrak{p}$. The proof ends the same way as the proof of Theorem \ref{theoprincipalthinHitchin} for $G=\Sp(2n,\mathds{R})$.

Let $k\equiv0,3[4]$ and $F$ be a totally real number field. We want to show that $\Lambda_F=\SO(q,\mathcal{O}_F)$ (see the \S\ref{subsectionexamplesarithmeticsubgroups} for the definition) contains infinitely many mapping class group orbits of Zariski-dense Hitchin representations. Let $A$ to be any quaternion algebra over $F$ that splits exactly at the real place where $q$ is not positive definite. Let $\mathcal{O}$ be an order in $A$ and denote $S=\mathds{H}^2/\mathcal{O}^1$. From now on, the proof goes as in the case $k\equiv1,2[4]$.
\end{proof}

\bibliographystyle{alpha}
\bibliography{main}

\end{document}